\documentclass[11pt,reqno]{amsart}
\usepackage[text={155mm,230mm},centering]{geometry}
\usepackage[mathscr]{eucal}
\usepackage[all]{xy}
\usepackage{mathrsfs}
\usepackage{amsfonts}
\usepackage{amsmath}
\usepackage{amsthm}
\usepackage{amssymb}
\usepackage{latexsym}
\usepackage{color}
\usepackage{graphicx}
\usepackage{float} 
\usepackage{subfigure}
\usepackage{stmaryrd}
\usepackage{pdfpages}
\usepackage[normalem]{ulem}
\usepackage{hyperref}
\newtheorem{theorem}{Theorem}[section]
\newtheorem{lemma}[theorem]{Lemma}

\newtheorem{corollary}[theorem]{Corollary}
\newtheorem{proposition}[theorem]{Proposition}
\newtheorem{proposition-definition}[theorem]{Proposition-Definition}
\newtheorem{example-proposition}[theorem]{Example-Proposition}

\newtheorem{question}[theorem]{Question}

\theoremstyle{definition}
\newtheorem{example}[theorem]{Example}
\newtheorem{definition}[theorem]{Definition}

\newtheorem{remark}[theorem]{Remark}
\newtheorem{remarks}[theorem]{Remarks}

\newtheorem*{ack}{Acknowledgements}

\newcommand{\bbC}{\mathbb{C}}
\newcommand{\bbZ}{\mathbb{Z}}
\newcommand{\gl}{\mathfrak{gl}}

\newcommand{\be}{\mathbf{e}}
\newcommand{\GL}{\mathrm{GL}}

\newcommand{\tos}{\mapsto}

\newcommand{\rmG}{\mathrm{G}}

\newcommand{\bm}{ \mathbf{m}}

\newcommand{\Hom}{\mathrm{Hom}}
\newcommand{\End}{\mathrm{End}}

\newcommand{\Alg}{ {\rm Alg}}
\newcommand{\Rep}{\mathrm{Rep}}
\newcommand{\Tr}{\mathrm{Tr} \,}
\newcommand{\bd}{\mathbf{d}}

\newcommand{\mD}{\mathcal{D}}

\newcommand{\QQ}{\mathbb{C} \overline{Q}}
\newcommand{\DQ}{\overline{Q}}
\newcommand{\ldb}{\mathopen{\{\!\!\{}}
\newcommand{\rdb}{\mathclose{\}\!\!\}}}
\newcommand{\bw}{\mathbf{w}}

\newcommand{\qneck}{{\mathbb{C} \overline{Q}}_{\hbar}}

\newcommand{\into}{\hookrightarrow}
\newcommand{\Spec}{\mathrm{Spec}\,}

\newcommand{\fkg}{\mathfrak{g}}

\newcommand{\fg}{\mathfrak{g}}
\newcommand{\HH}{\mathsf{HH}}
\newcommand{\br}{\mathbf{r}}

\newcommand{\id}{\mathrm{id}}
\newcommand{\rmH}{\mathrm{H}}

\newcommand{\cyc}{\mathrm{cyc}}
\newcommand{\tr}{\mathrm{tr}}
\newcommand{\Sym}{\operatorname{Sym}}

\newcommand{\Com}{\mathsf{Com}}
\newcommand{\fkh}{\mathfrak{h}}
\newcommand{\hc}{\mathcal{HC}}
\newcommand{\reg}{\mathrm{reg}}

\newcommand{\sph}{{\mathrm{sph}}}

\newcommand{\gr}{\mathrm{gr}\,}
\newcommand{\cc}{\circ}

\newcommand{\loc}{{\mathrm{loc}}}
\newcommand{\rmT}{\mathrm{T}}

\newcommand{\bfl}{\mathbf{l}}

\newcommand{\pa}{\partial}

\newcommand{\Ord}{\operatorname{Ord}}
\newcommand{\mC}{\mathcal{C}}
\newcommand{\ind}{\mathbf{1}}
\newcommand{\bfi}{\mathbf{i}}
\newcommand{\mM}{\mathcal{M}}

\newcommand{\pr}{\mathrm{pr}}
\newcommand{\tTr}{\widetilde{\Tr}}
\newcommand{\pol}{\mathrm{pol}}
\makeatletter

\newcommand{\Rmnum}[1]{\expandafter\@slowromancap\romannumeral #1@}
\makeatother

\allowdisplaybreaks
\title[Quantization of Preprojective Algebras]{Quantization
of Preprojective Algebras and the Type A
Rational Cherednik Algebra}

\author{Meiliang Liu}
\author{Jun Zhang}
\address{Liu, Zhang: 
School of Mathematics, 
Sichuan University, Chengdu, 
Sichuan Province  610064, P.
R. China 
}
\email{liumeiliang@stu.scu.edu.cn,
zhangjun24@stu.scu.edu.cn
}

\author{Hu Zhao}
\address{Zhao (corresponding author): 
School of Mathematics and Statistics, Hainan University, 
Haikou 570228, P.
R. China}
\email{zhmath@hainanu.edu.cn}
\date{}


\begin{document}

\begin{abstract}{
For the Jordan quiver, we construct a radial 
quantum trace morphism from the quantum 
preprojective algebra to the algebra of invariant 
differential operators on the Cartan subalgebra. 
Its classical limit recovers the classical trace 
morphism associated with the commuting quotient, 
and we prove that both the quantum and classical 
trace morphisms are surjective.
We further relate noncommutative quantizations of 
the Jordan preprojective algebra to the spherical 
rational Cherednik algebra of type~A. Finally, we 
derive explicit formulas for the Harish--Chandra 
homomorphism
and its radial part in power-sum coordinates.}
\end{abstract}

\maketitle
\tableofcontents
\section{Introduction}

 Quiver varieties play fundamental roles in mathematical physics and geometric representation theory, and the Jordan quiver provides a minimal but particularly important example.
In the unframed setting, it yields the commuting quotient; in the framed setting, the corresponding quiver varieties lead to Hilbert schemes of points on the plane and to Calogero--Moser spaces \cite{Nak1994Ins,Nak1998Qui,EG2002}.
In recent years,
their quantization has become a central problem in symplectic duality \cite{BLPB2016II}.

 Since quiver varieties arise by Hamiltonian reduction, the standard approach to their quantization begins with the Weyl algebra of the representation space and imposes the quantum moment-map relations. This principle is commonly expressed by the phrase ``quantization commutes with reduction.'' Holland proved the corresponding statement for quiver varieties under a flatness hypothesis on the moment map \cite{Hol1999}, which generally fails for the Jordan quiver.

 {On the other hand, a finite quiver \(Q\) determines a preprojective algebra \(\Pi Q\), and the associated quiver varieties may be viewed as moduli spaces of \(\Pi Q\)-representations. The Kontsevich--Rosenberg principle suggests a close relationship between the geometry of these moduli spaces and the noncommutative geometry of the preprojective algebra.}

 On the noncommutative side, no analogous flatness hypothesis is required. Schedler constructed a quantization of the necklace Lie algebra together with quantum trace maps \cite{Sch2005,GinSch2006}. Building on this construction, the last author proved that noncommutative quantization commutes with noncommutative Hamiltonian reduction for quiver algebras\cite{Zhao2021Com,Zhao2023Non}. 
These results are compatible with the Kontsevich--Rosenberg principle and 
reveal a hidden dominance of the noncommutative side over the representation-side quantization.

The failure of flatness in the Jordan case prevents a direct application of the ``quantization commutes with reduction'' theorem and motivates studying the radial-part comparison directly.

\subsection{Main results}

Let $Q$ be the Jordan
quiver, which consists
of one vertex and one
loop. The moment map
on the quiver
representations is
not flat for \(n>1\); nevertheless,
one has \(\mM_n (Q) \cong \rmT^\ast \bbC^n // S_{n}\), and the latter quotient is quantized by the invariant differential-operator algebra \(\mD_\hbar (\bbC^n)^{S_n}\).
On the noncommutative side, there is an algebra \(\Pi Q_{\hbar}\) intended to give quantizations of the
same classical preprojective algebra \(\Pi Q\).

This leads to the following natural question.

\begin{question}\label{que: [q, r] = 0}
For the Jordan quiver \(Q\), does there
exist a canonical morphism
\[
\Pi Q_{\hbar}\longrightarrow \mD_\hbar(\bbC^n)^{S_n}?
\]
\end{question}

By combining Schedler's quantum trace map with the Harish--Chandra homomorphism, we construct a canonical morphism \(\widetilde{\Tr^q}\) and its semiclassical limit \(\tTr\). The following theorem answers Question \ref{que: [q, r] = 0}, proved as Proposition \ref{prop: tilt Tr semiclassical} and Proposition \ref{prop: tTr surj}.

\begin{theorem}
\label{thm: intro qtrace suj}
Let \(Q\) be the Jordan quiver and let \(n\geq1\).
\begin{enumerate}
\item[$(1)$]
The radial quantum trace is a surjective
\(\bbC[[\hbar]]\)-algebra homomorphism
\[
\widetilde{\Tr^q}:
\Pi Q_{\hbar}
\longrightarrow
\mD_\hbar(\fkh)^{S_n}.
\]

\item[$(2)$]
Its semiclassical limit is the surjective Poisson algebra homomorphism
\[
\widetilde{\Tr}:
\Sym((\Pi Q)_\cyc)
\longrightarrow
\bbC[T^*\fkh]^{S_n}
\]
determined by
\[
\widetilde{\Tr}([x^a y^b])
=
\sum_{i=1}^n x_i^a y_i^b.
\]
\end{enumerate}
\end{theorem}

Surjectivity is an essential part of the result: it shows that every invariant differential operator on the Cartan subalgebra arises from the noncommutative quantum reduction through \(\widetilde{\Tr ^q}\).
Classically, the corresponding statement is the generation of the multisymmetric invariant ring by polarized power sums.

The above theorem leaves a second, independent problem: although \(\widetilde{\Tr ^q},\ \tTr\) are structurally well understood, the radial parts of the mixed invariant differential operators arising from quantum traces are difficult to compute.
The classical results of
Harish--Chandra, Wallach and Levasseur--Stafford describe the kernel and image of the Harish--Chandra
homomorphism \cite{HC1957Dif,HC1964Inv,Wal1993Inv,LevSta1995Inv,LevSta1996Ker}; they do
not, by themselves, rewrite a mixed operator in a form from which its radial
part can be read off.  The basic operators arising from the quantum trace
map are precisely of this mixed type.

Let
\[
X=(x_{ij}),
\qquad
D=(\partial_{x_{ji}})
\]
be respectively the matrix of coordinate functions and the matrix of
\(\hbar\)-Weyl generators on \(\gl_n (\bbC)\).
For \(s,r\geq0\), the results of Schedler \cite{Sch2005} and Zhao \cite{Zhao2021Com,Zhao2023Non} show that the normal-ordered operators \(\Tr^q(X^s D^r)\) are basic building blocks of the quantum Hamiltonian reduction (see Section \ref{subsec: state radial part in powersum}).

Let $\hc$ denote the Harish--Chandra homomorphism and $\hc'$ its radial
part.  Set \(p_a=\sum_{i=1}^n x_i^a\) and
\(\delta=\prod_{i<j}(x_i-x_j)\), and 
let
\(\partial_{p_a}\), for \(1\leq a\leq n\), denote the corresponding
\(\hbar\)-Weyl algebra generators for the coordinates
\(p_1,\ldots,p_n\).
The following theorem combines Theorem \ref{thm: radial part}
with Proposition \ref{prop: hc formula power sums}.

\begin{theorem}[Harish--Chandra formula and its radial part in power sums]
\label{thm:intro-radial-formula}
Let \(\fkg=\gl_n(\bbC)\), and let \(\fkh\) be its Cartan subalgebra
of diagonal matrices. For every \(s\in\bbZ_{\geq0}\) and
\(r\in\bbZ_{\geq1}\),
\[
\hc'\bigl(\Tr^q(X^sD^r)\bigr)
=
\sum_{(\pi,\alpha)}
\hbar^{r-|\pi|}
C_{\pi,\alpha}^{s,r}(p)\,
\partial_p^\alpha,
\]
and
\[
\hc\bigl(\Tr^q(X^sD^r)\bigr)
=
\sum_{(\pi,\alpha)}
\hbar^{r-|\pi|}
C_{\pi,\alpha}^{s,r}(p)
\prod_{B\in\pi}
\left(
\partial_{p_{\alpha(B)}}
-
\frac12
\partial_{p_{\alpha(B)}}(\ln\delta^2)
\right).
\]
In particular, every
coefficient \(C_{\pi,\alpha}^{s,r}(p) \) is a polynomial in the power sums \(p_k\).
\end{theorem}

Note that power-sum realizations of differential operators are
well established,
see \cite{SerVes2015,Dra2024Gen} for examples.
The new point of the above theorem is a closed-form,
finite-\(n\), and nonrecursive power-sum expansion for arbitrary \(s\) and \(r\), in which coefficients \(C_{\pi,\alpha}^{s,r} (p)\) are
characterized by combinatorial data such as set partitions and graphs.

More broadly,
these formulas provide a concrete starting point and may reveal a structural connection between the noncommutative geometry on quiver algebras
and the deformed \(W_{1+\infty}\), affine-Yangian, and
Calogero--Moser systems \cite{ArbSch2013,KN2016Qua,SchVas2013,DesHal2012}.

The target in Question
\ref{que: [q, r] = 0} is the algebra of invariant 
differential operators on the Cartan subalgebra. 
The type~A rational Cherednik algebras of Etingof--Ginzburg
\cite{EG2002} provide a related family.
For fixed \(c\), the spherical algebra
\(\be\rmH_{\hbar,c}(n)\be\) is a formal deformation of the
commutative algebra \(\be\rmH_{0,c}(n)\be\).
After localization at \(\delta\), the spherical Dunkl
isomorphism identifies this deformation with
\(\mD_\hbar(\fkh_{\reg})^{S_n}\)\cite{EG2002,Eti2009Lec}.
This raises a further natural question.

\begin{question}\label{que: rca}
Does there exist a canonical morphism 
\[
\Pi Q_{\hbar}
\longrightarrow
\be\rmH_{\hbar,c}(n)^{\loc}\be
\]
to the localized (formal) spherical rational Cherednik algebra?
\end{question}

Our answer is as follows,
which is proved as Proposition
\ref{prop: quant preproj to RCA} and Corollary \ref{cor: preproj vs rca}.

\begin{theorem}
\label{thmfrompreprojintro}
Let \(Q\) be the Jordan quiver and let \(n\geq1\).  For every
\(c\in\bbC\), there is a \(\bbC[[\hbar]]\)-algebra homomorphism
\[
\Upsilon_{\hbar,c}:
\Pi Q_{\hbar}
\longrightarrow
\be\rmH_{\hbar,c}(n)^\loc\be
\]
which factors through the quantum trace map \(\Tr^q\);
its semiclassical limit is a Poisson algebra morphism
\[
\Upsilon_{0,c}:
\Sym((\Pi Q)_\cyc)
\longrightarrow
\be\rmH_{0,c}(n)^\loc\be
\]
which factors through the classical trace map \(\Tr\).
When \(c=0\), their images lie in the nonlocalized spherical subalgebra.
\end{theorem}

Theorem
\ref{thm:intro-radial-formula} expresses the images of operators in power-sum coordinates;
it is natural to ask for a corresponding formula for \(\Upsilon_{\hbar,c}\)
in the same coordinates.
Set
\[
G(p)_{a,b}:=p_{a+b-2}, \quad \Xi_{p_a}
:=
\frac1a\sum_{b=1}^n
\bigl(G(p)^{-1}\bigr)_{a,b}
\be\left(\sum_{i=1}^n x_i^{b-1}y_i\right)\be.
\]
We compute the inverse spherical Dunkl isomorphism in the same power-sum coordinates and then derive an explicit formula for \(\Upsilon_{\hbar, c}\); these results appear as Proposition \ref{prop: inverse dunkl power sum} and Theorem \ref{thm: Upsilon in power-sum}.

\begin{theorem}
	\begin{enumerate}
		\item For every \(c\in\bbC\), the assignment
		\begin{equation*}
			p_a\mapsto p_a\be,
			\qquad
			\delta^{-2}\mapsto\delta^{-2}\be,
			\qquad
			\partial_{p_a}\mapsto\Xi_{p_a},
			\qquad 1\leq a\leq n,
		\end{equation*}
		extends uniquely to a  \(\bbC[[\hbar]]\)-algebra morphism
		\[
		\Phi_{\hbar,c}:
		\mD_\hbar(\fkh_{\reg})^{S_n}
		\longrightarrow
		\be\rmH_{\hbar,c}(n)^\loc\be.
		\]
Furthermore, it is the inverse of the spherical Dunkl isomorphism in Proposition \ref{prop: dunkl embed loc neq 0}.

\item For \(c\in\bbC\), \(s\in\bbZ_{\geq0}\), and
\(r\in\bbZ_{\geq1}\),
\begin{equation*}
\Upsilon_{\hbar,c}(\overline w_{s,r})
=\sum_{(\pi,\alpha)}
			\hbar^{r-|\pi|}C_{\pi,\alpha}^{s,r}(p)\be
			\prod_{B\in\pi}
			\left(
			\Xi_{p_{\alpha(B)}}
			-
			\frac{1}{2}
			\partial_{p_{\alpha(B)}}(\ln \delta^2)\be
			\right).
\end{equation*}
For \(r=0\),
\[
\Upsilon_{\hbar,c}(\overline w_{s,0})
=p_s\be.
\]
\end{enumerate}
\end{theorem}
Here, \(\overline w_{s,r}\)
are topological generators of \(\Pi Q_{\hbar}\).

\subsection{Organization
of the paper}

The rest of the paper is organized
as follows.
Section \ref{sect:NPG} reviews double Poisson brackets, noncommutative
Hamiltonian reduction, and their realization on representation spaces,
with particular attention to doubled quivers.  Section
\ref{sec: noncom Ham red} recalls Schedler's quantum necklace algebra,
Holland's representation-side quantum Hamiltonian reduction, and the
noncommutative quantization-reduction theorem; it also isolates the
Jordan-quiver flatness obstruction.  Section \ref{sec: quant mor} constructs
\(\widetilde{\Tr^q}\), proves the classical and quantum surjectivity
statements, and derives the radial-part and Harish--Chandra formulas in
power-sum coordinates.  Section \ref{sec: comparison} constructs
\(\Upsilon_{\hbar,c}\), computes the inverse spherical Dunkl isomorphism in
power-sum coordinates, and gives explicit formulas for the images of quantum
trace operators in the localized spherical rational Cherednik algebra.
Appendix \ref{app: fdb} records the multivariate Fa\`a di Bruno formula used
in the radial-part computation, and Appendix \ref{sect:examples} works out
the graph-to-power-sum mechanism in representative examples.

\begin{ack}
We thank Xiaojun Chen and Zhuyang Li for helpful discussions.
This work was partially supported by the National Natural Science Foundation of China under Grants No.~12271377 and No.~12531003.
\end{ack}

\section{Noncommutative Poisson geometry}
\label{sect:NPG}

The frameworks of noncommutative Poisson and symplectic geometry used in this paper were developed by Van den Bergh and by Crawley--Boevey--Etingof--Ginzburg in \cite{Van2008Double,CBEG2007}, respectively.
This section recalls the definitions and results needed later.

\subsection{Noncommutative Hamiltonian reduction}
Let \(I\) be a finite index set and set
\[
R=\bigoplus_{i\in I}\bbC e_i,
\]
where the \(e_i\) are pairwise orthogonal idempotents.
We abbreviate \(-\otimes_\bbC-\) by \(-\otimes-\).

\begin{definition}[Van den Bergh]
\label{def: dpois bracket}
Let \(A\) be an \(R\)-algebra. A \emph{double Poisson bracket} on \(A\) is an \(R\)-bilinear map
\[
\ldb-,-\rdb\colon A\otimes A\to A\otimes A
\]
such that, for all \(a,b,c\in A\),
\begin{enumerate}
\item \(\ldb a,b\rdb=-\ldb b,a\rdb^\circ\);

\item \(\ldb a,bc\rdb=\ldb a,b\rdb c+b\ldb a,c\rdb\);

\item \(\ldb a,\ldb b,c\rdb\rdb_L+(132)\ldb c,\ldb a,b\rdb\rdb_L+(123)\ldb b,\ldb c,a\rdb\rdb_L=0\).
\end{enumerate}
\end{definition}

Here, we write
\({\ldb b,a\rdb\in A\otimes A}\)
as \(\ldb b,a\rdb'\otimes\ldb b,a\rdb''\), with summation understood, and set
\(
\ldb b,a\rdb^\circ=\ldb b,a\rdb''\otimes\ldb b,a\rdb'
\).
Moreover, \(\ldb a,-\rdb_L\) denotes the operation obtained by applying the bracket to the first tensor factor;
and \((123)\) and \((132)\) act by permuting the three tensor factors.
{A crucial fact is that a double Poisson bracket on \(A\)
induces a Lie bracket on the commutator quotient.
Set
\[
\{a,b\}:=\bm(\ldb a,b\rdb)=\ldb a,b\rdb'\,\ldb a,b\rdb'' ,
\]
where $\bm$ is the multiplication on $A$.}

\begin{proposition}[{\cite[Corollary 2.4.6]{Van2008Double}}]
\label{prop: induced Lie bracket}
Let \((A,\ldb-,-\rdb)\) be a double Poisson algebra. Then the commutator quotient space
\(A_\cyc:=A/[A,A]\) inherits a Lie algebra structure via the bracket \(\{-,-\}\).
\end{proposition}

The commutator quotient \(A_\cyc=A/[A,A]\) is \(\HH_0(A)\).  The Lie bracket above is the structure
induced by the chosen double Poisson bracket, and it is the
noncommutative Poisson structure used throughout this paper.

We now turn to noncommutative Hamiltonian reduction.

\begin{definition}[Crawley--Boevey--Etingof--Ginzburg, Van den Bergh]\label{def: nc moment map}
Let \((A,\ldb-,-\rdb)\) be a double Poisson algebra. A \emph{noncommutative moment map} is an element \(\bw=\sum_{i\in I}\bw_i\in\bigoplus_{i\in I}e_iAe_i\) such that, for every \(p\in A\),
\[
\ldb\bw,p\rdb=\sum_{i\in I}\bigl(pe_i\otimes e_i-e_i\otimes e_ip\bigr).
\]
\end{definition}

A \emph{noncommutative Hamiltonian space} is a double Poisson algebra endowed with a noncommutative moment map.

\begin{definition}[Crawley--Boevey--Etingof--Ginzburg, Van den Bergh]\label{def: nc ham red}
For a noncommutative Hamiltonian space \((A,\ldb-,-\rdb,\bw)\), its \emph{noncommutative Hamiltonian reduction} is the quotient algebra
\[
A_\bw:=A/A\bw A.
\]
\end{definition}

The following proposition is the 
noncommutative analogue of the fact that Hamiltonian 
reduction produces a Poisson structure on the reduction 
space.

\begin{proposition}
[{\cite[Proposition~2.6.5]{Van2008Double}, \cite[Theorem 7.2.3]{CBEG2007}}]
\label{prop: noncom red}
Let \((A,\ldb-,-\rdb,\bw)\) be a noncommutative 
Hamiltonian space.
Then the Lie bracket on \(A_\cyc\) descends to 
\((A_\bw)_\cyc\), and the projection 
\(A_\cyc\to(A_\bw)_\cyc\) is a Lie algebra homomorphism.
\end{proposition}

Whenever an algebra \(B\) is obtained from \(A\) by a Hamiltonian-reduction procedure, we use the notation
\[
\xymatrix{A\ar@{-->}[r]&B.}
\]

\subsection{Representation spaces
and Kontsevich--Rosenberg principle}

The Kontsevich--Rosenberg principle asserts that a noncommutative geometric structure on an associative algebra should induce its classical counterpart on the algebra's representation spaces.
To formulate the instances needed here, let \(A\) be a noncommutative \(R\)-algebra and let \(V\) be a finite-dimensional \(R\)-module.
Its representation space is
\[
\Rep^A _V:=\Hom_{\Alg_R} \bigl(A,\End_{\bbC}(V)\bigr),
\]
{with the conjugation action of \( \mathrm{Aut}_R (V)\).}
In this paper, we take
\[
V=\bbC^\bd:=\bigoplus_{i\in I}\bbC^{d_i}
\]
for a dimension vector \(\bd=(d_i)\), and write \(\Rep^A_\bd\).
In this case, \({\rm Aut}_R (V) = \GL_\bd(\bbC):=\prod_{i\in I}\GL_{d_i}(\bbC)\), with Lie algebra \(\gl_\bd(\bbC):=\bigoplus_{i\in I}\gl_{d_i}(\bbC)\).
For \(a\in A\), the matrix coefficients
\[
(a)_{p,q}\colon \Rep^A_\bd\to \bbC,
\qquad
\rho\tos \rho(a)_{p,q},
\]
will be denoted by \(a_{p,q}\).
The coordinate ring \(A_\bd:=\bbC[\Rep^A_\bd]\) is generated by these matrix coefficients; see \cite[Proposition~12.1.6]{Gin2005non}.
An important notion is the {classical trace map}:
\[
\Tr:
A
\longrightarrow
A_\bd,
\qquad
a\longmapsto
\bigl( \rho \longmapsto \operatorname{tr} ( \rho(a) )\bigr).
\]
It is clear that \(\Tr\) vanishes on commutators and descends to \(A_\cyc\).

The following theorem records the compatibility of Definitions \ref{def: dpois bracket}, \ref{def: nc moment map}, and \ref{def: nc ham red} with the Kontsevich--Rosenberg principle.

\begin{theorem}
\begin{enumerate}
\item[$(1)$]
If \((A,\ldb-,-\rdb)\) is a double Poisson algebra, then, for every dimension vector \(\bd\), \(\mathrm{Rep}_{\mathbf{d}}^A\) is a Poisson scheme. The induced Poisson bracket on \(A_\bd\) is given by
\begin{equation*}
\{a_{ij},b_{uv}\}=\ldb a,b\rdb'_{u,j}\,\ldb a,b\rdb''_{i,v},
\end{equation*}
for \(a,b\in A\).

\item[$(2)$]
If, in addition, \((A,\ldb-,-\rdb,\bw)\) is a noncommutative Hamiltonian space, then \(\Rep^A_\bd\) is a Poisson \(\GL_\bd(\bbC)\)-space with moment map
\[
\mu\colon\Rep^A_\bd\to\gl_\bd(\bbC)^*,
\qquad
\rho\tos\tr\bigl(\rho(\bw)\,\cdot-\bigr).
\]

\item[$(3)$]
One has \(\Rep^{A_\bw}_\bd=\mu^{-1}(0)\) as affine schemes. Consequently,
\[
\Rep^{A_\bw}_\bd//\GL_\bd(\bbC)
=
\mu^{-1}(0)//\GL_\bd(\bbC),
\]
that is, the categorical quotient on the left is the Hamiltonian reduction of \(\Rep^A_\bd\) at zero.
\end{enumerate}
\end{theorem}
The  {Poisson bracket in part~\textup{(1)} is
\cite[Proposition~7.5.2]{Van2008Double}, and the moment-map statement
in part~\textup{(2)} is \cite[Proposition~7.11.1]{Van2008Double};
part~\textup{(3)} follows directly from the definitions.  See also
\cite[Theorem~6.4.3]{CBEG2007} in the bisymplectic setting}.

\subsection{Example:
Double quivers}\label{subsec: necklace}

In this work, given a finite quiver \(Q\) and an arbitrary arrow \(a\in Q_1\), we write \(s(a)\) for its source and \(t(a)\) for its target and \(a = e_{t(a)} a e_{s(a)}\).
{We identify a vertex} {\(i\)} {with its length-zero path}
{\(e_i\)}.

Let \(\DQ\) be the double quiver of \(Q\), obtained by adjoining a reverse arrow \(a^*\) for each \(a\in Q_1\).
For arrows \(a,b\in\overline{Q}_1\), set
\[
\varepsilon(a,b)
=
\begin{cases}
  1,& a\in Q_1\text{ and }b=a^*,\\
 -1,& b\in Q_1\text{ and }a=b^*,\\
  0,& \text{otherwise}.
\end{cases}
\]

Van den Bergh's double Poisson bracket on \(\QQ\) is determined by
\begin{equation}\label{Vdbsdoublebracket}
\ldb a,a^*\rdb=e_{s(a)}\otimes e_{t(a)},
\qquad
\ldb a^*,a\rdb=-e_{t(a)}\otimes e_{s(a)},
\end{equation}
{for \(a\in Q_1\). { Extend the reverse-arrow involution by}
{  \((a^*)^*=a\)}.
{ For arrows} {  \(f,g\in\overline Q_1\)}  {with}
{  \(f\neq g^*\)}, { set} {  \(\ldb f,g\rdb=0\)}, and  {extend these
generator rules} by  {bilinearity, skew-symmetry, and the double Leibniz
identity}.} The induced bracket on \(\QQ_\cyc\), called the necklace Lie algebra, satisfies
\begin{equation}\label{necklace bracket}
\begin{aligned}
&\{[a_1a_2\cdots a_k],[b_1b_2\cdots b_l]\}\\
&=\sum_{\substack{1\leq i\leq k\\1\leq j\leq l}}
\varepsilon(a_i,b_j)\,
\Bigl[
t(a_{i+1})a_{i+1}\cdots a_ka_1\cdots a_{i-1}
b_{j+1}\cdots b_lb_1\cdots b_{j-1}
\Bigr].
\end{aligned}
\end{equation}
for cyclic paths \([a_1\cdots a_k]\) and \([b_1\cdots b_l]\) in \(\QQ_\cyc\), where the indices are taken modulo \(k\) and \(l\).
The following proposition applies Proposition 
\ref{prop: noncom red} to doubled quivers.

\begin{proposition}[{\cite{CBEG2007, Van2008Double}}]\label{prop: noncom Hamil on quiver}
Let \(Q\) be a finite quiver.
Then \(\bw=\sum_{a\in Q_1}(aa^*-a^*a)\) is a
noncommutative moment map for \((\QQ,\ldb-,-\rdb)\), 
and the preprojective algebra
\[
\Pi Q=\QQ/\QQ\bw\QQ
\]
is obtained from \(\QQ\) by noncommutative Hamiltonian reduction. Consequently, \((\Pi Q)_\cyc\) is a Lie algebra and \(\QQ_\cyc\to(\Pi Q)_\cyc\) is a Lie algebra morphism.
\end{proposition}

\begin{proof}
This is the quiver moment-map construction of
\cite[Theorem~6.3.1]{Van2008Double} and
\cite[Proposition~8.1.1(ii)--(iii), Section~8.2]{CBEG2007}.
\end{proof}

\begin{example}\label{eg: nc pois on Jordan}
For the Jordan quiver, \(\QQ\cong\bbC\langle x,y\rangle\), and
\[
\ldb x,y\rdb=1\otimes1,
\qquad
\ldb y,x\rdb=-1\otimes1,
\qquad
\ldb x,x\rdb=\ldb y,y\rdb=0.
\]
The moment map is \(\bw=xy-yx\), so \(\Pi Q\cong\bbC[x,y]\); the induced bracket is the standard Poisson bracket on \(\bbC[x,y]\).
\end{example}

\section{Noncommutative reduction commutes with quantization}\label{sec: noncom Ham red}

This section reviews the noncommutative form of ``quantization commutes with reduction'' for quiver algebras and explains the flatness issue that appears after passage to representation spaces.
The distinction between noncommutative reduction and representation-side Hamiltonian reduction is essential: the former is defined for every finite quiver, whereas Holland's work identifies the latter with a quantization of the classical quiver variety only when the classical moment map is flat.

\subsection{Quantization of double quivers}
\label{subsec: nc quant}

Building on \cite{Tur1991Ske}, Schedler \cite{Sch2005} constructed a
Hopf-algebra quantization of the necklace Lie algebra.
Let \(Q\) be a finite quiver, and let
\[
R=\bigoplus_{i\in Q_0}\bbC e_i
\]
be the semisimple algebra generated by the orthogonal idempotents. We use the following notation.
\begin{enumerate}
\item Let \(AH:=\overline{Q}_1\times\mathbb{N}\), {the set of arrows with heights}.

\item Let \(E_{\overline{Q},h}\) be the \(\mathbb{C}\)-vector space spanned by \(AH\), {equipped with the natural \(R\)-bimodule structure determined by the source and target of each arrow.}

\item Let \(LH:=(T_RE_{\overline{Q},h})_\cyc\), {the generalized cyclic path space with heights}.

\item Let \(SLH_{\mathrm{pol}}:=\operatorname{Sym}(LH)[\hbar]\). The symmetric product is denoted by \(\&\).
\end{enumerate}

Let \(SLH'_{\mathrm{pol}}\) be the \(\mathbb{C}[\hbar]\)-submodule
spanned by elements of the form
\begin{equation}
\label{general form}
\begin{split}
(a_{1,1},h_{1,1})\cdots&(a_{1,l_1},h_{1,l_1})
\&
(a_{2,1},h_{2,1})\cdots(a_{2,l_2},h_{2,l_2})
\\
&\quad \&\cdots\&
(a_{k,1},h_{k,1})\cdots(a_{k,l_k},h_{k,l_k})
\&v_1\&v_2\&\cdots\&v_m.
\end{split}
\end{equation}
Here the \(h_{i,j}\) are pairwise distinct, \(a_{i,j}\in\overline Q_1\), and \(v_i\in Q_0\).
Let \(\widetilde A_{\mathrm{pol}}\) be the quotient of
\(SLH'_{\mathrm{pol}}\) obtained by identifying two height assignments
whenever they induce the same relative order on all heights.

Next, let \(\widetilde B_{\mathrm{pol}}\) be the
\(\mathbb{C}[\hbar]\)-submodule of \(\widetilde A_{\mathrm{pol}}\)
generated by the following skein relations.
\begin{itemize}
\item \(X-X^{'}_{i,j,i^{'},j^{'}}+\hbar X^{''}_{i,j,i^{'},j^{'}}\), \\[2mm]
where \(i\neq i^{'}\), \(h_{i,j}<h_{i^{'},j^{'}}\), and there is no \((i^{''},j^{''})\) such that
\(h_{i,j}<h_{i^{''},j^{''}}<h_{i^{'},j^{'}}\);\vspace{2mm}

\item \(X-X^{'}_{i,j,i,j^{'}}+\hbar X^{''}_{i,j,i,j^{'}}\), \\[2mm]
where \(h_{i,j}<h_{i,j^{'}}\), and there is no \((i^{''},j^{''})\) such that
\(h_{i,j}<h_{i^{''},j^{''}}<h_{i,j^{'}}\).
\end{itemize}
In these relations, \(X'\) and \(X''\) are defined as follows.
The element \(X^{'}_{i,j,i^{'},j^{'}}\) is obtained from \(X\) by interchanging the heights \(h_{i,j}\) and \(h_{i^{'},j^{'}}\).
The element \(X^{''}_{i,j,i^{'},j^{'}}\) is obtained by replacing the two components
\[
(a_{i,1},h_{i,1})\cdots(a_{i,l_i},h_{i,l_i})
\quad\text{and}\quad
(a_{i^{'},1},h_{i^{'},1})\cdots(a_{i^{'},l_{i^{'}}},h_{i^{'},l_{i^{'}}})
\]
with the single component
\begin{displaymath}
	\varepsilon(a_{i,j},a_{i^{'},j^{'}})t(a_{i,j+1})(a_{i,j+1},h_{i,j+1})\cdots(a_{i,j-1},h_{i,j-1})(a_{i^{'},j^{'}+1},h_{i^{'},j^{'}+1})\cdots(a_{i^{'},j^{'}-1},h_{i^{'},j^{'}-1}).
\end{displaymath}
Similarly, \(X^{'}_{i,j,i,j^{'}}\) is obtained from \(X\) by interchanging the heights \(h_{i,j}\) and \(h_{i,j^{'}}\), whereas \(X^{''}_{i,j,i,j^{'}}\) is obtained by replacing the component \((a_{i,1},h_{i,1})\cdots(a_{i,l_i},h_{i,l_i})\) with
\begin{align*}
\varepsilon(a_{i,j}, a_{i,j^{'}} )t(a_{i,j^{'}+1})(a_{i,j^{'}+1},h_{i,j^{'}+1})\cdots(a_{i,j-1},h_{i,j-1})\\
\&\ 
t(a_{i,j+1})(a_{i,j+1},h_{i,j+1})\cdots(a_{i,j^{'}-1},h_{i,j^{'}-1}).
\end{align*}

Set first
\[
\qneck^{\mathrm{pol}}
:=\widetilde A_{\mathrm{pol}}/\widetilde B_{\mathrm{pol}},
\]
and then define the formal algebra used throughout the paper by
\[
\qneck
:=
\varprojlim_m
\qneck^{\mathrm{pol}}/\hbar^m\qneck^{\mathrm{pol}}.
\]
On \(\qneck^{\mathrm{pol}}\), the product \(X\ast Y\) is defined by
placing all heights occurring in \(Y\) above those occurring in \(X\);
it extends uniquely and continuously to \(\qneck\).

For the necklace Lie algebra \(\QQ_\cyc\),
Schedler proved that the PBW-type theorem holds for \(\qneck\) (see \cite[Section 4]{Sch2005}); then Ginzburg-Schedler gave a different version in \cite{GinSch2006}.
His PBW basis is constructed as follows.
Order the cyclic paths and vertex classes, and choose a first arrow in
each nonconstant cyclic path.  Lift every ordered symmetric monomial
in these classes by assigning consecutive heights along its successive
cyclic factors.  These lifted monomials form a
\(\bbC[\hbar]\)-basis before completion and a topological
\(\bbC[[\hbar]]\)-basis after completion.
Sending each such basis element to its height-free symmetric monomial
defines \(\pr\).  This map depends on the chosen basis conventions;
it is not literal height-forgetting on arbitrary representatives.
\begin{theorem}\cite[Theorem 4.1]{Sch2005}
\label{thm: schedler's}
Let \(Q\) be a finite quiver, and fix an order on the set \(\{x_i\}\) of cyclic paths in \(\overline Q\) and idempotents in \(Q_0\).
Then the above height-forgetting projection
establishes an isomorphism
\[
{\rm pr}\colon\qneck\longrightarrow \Sym(\QQ_\cyc)[[\hbar]]
\]
of topologically free  \(\mathbb{C}[[\hbar]]\)-modules.
Moreover, \(\qneck\) is a noncommutative quantization of \(\QQ\).
\end{theorem}

Here, given an algebra \(A\) with a noncommutative Poisson structure, a \emph{noncommutative quantization} of \(A\) means a flat, \(\hbar\)-adically complete \(\bbC[[\hbar]]\)-algebra \(A_\hbar\), together with a module isomorphism
\[
\ell_\hbar:\Sym(A_\cyc)[[\hbar]]\xrightarrow{\sim}A_\hbar.
\]
Its reduction modulo \(\hbar\) is required to be an isomorphism of unital algebras. Under this identification, the bracket induced on \(A_\hbar/\hbar A_\hbar\) by \(-\hbar^{-1}[-,-]\bmod\hbar\) must coincide with the Poisson bracket on \(\Sym(A_\cyc)\) obtained by extending the given Lie bracket on \(A_\cyc\);
in other words, for \(u,v\in\qneck\),
\begin{equation}\label{nc  pic}
\{\pr_0(u),\pr_0(v)\}
=-\frac{1}{\hbar}[u,v]\bmod\hbar,
\qquad
\pr_0(u):=u\bmod\hbar.
\end{equation}

\begin{example}
	Let $Q$ be the quiver in Figure \ref{A3};
	its double $\overline{Q}$ is shown in Figure \ref{DA3}.
	\begin{figure}[H]
		\centering
		\subfigure[$Q$]{
			\includegraphics[scale=0.4]{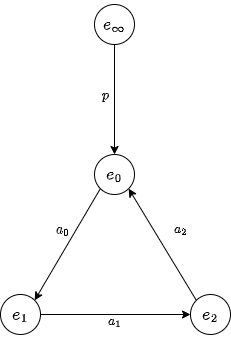}
			\label{A3}}
		\qquad\qquad\qquad
		\subfigure[$\overline Q$]{
			\includegraphics[scale=0.4]{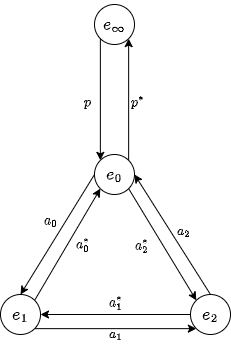}
			\label{DA3}}
		\caption{Quiver $Q$ and its doubled version $\overline{Q}$.}
	\end{figure}
	Consider
	\begin{equation*}
		X = (a_0 ^\ast, 1)(a_1 ^\ast, 2)(a_2 ^\ast, 3)\ \text{and}\ Y = (a_2 ^\ast, 1)(a_2, 2).
	\end{equation*}
	Then
	\begin{equation*}
		\begin{split}
			[X, Y] & = (a_0 ^\ast, 1)(a_1 ^\ast, 2)(a_2 ^\ast, 3) \& (a_2 ^\ast, 4)(a_2 , 5) - (a_2 ^\ast, 1)(a_2, 2) \& (a_0 ^\ast, 3)(a_1 ^\ast, 4)(a_2 ^\ast, 5)\\
			& = (a_0 ^\ast, 2)(a_1 ^\ast, 3)(a_2 ^\ast, 4) \& (a_2 ^\ast, 1)(a_2 , 5) - (a_2 ^\ast, 1)(a_2, 2) \& (a_0 ^\ast, 3)(a_1 ^\ast, 4)(a_2 ^\ast, 5)\\
			& = - \hbar \varepsilon(a_2 ^\ast, a_2)(a_0 ^\ast, 2)(a_1 ^\ast, 3)(a_2 ^\ast, 1)\\
			& = - \hbar \varepsilon(a_2 ^\ast, a_2)(a_0 ^\ast,1)(a_1 ^\ast, 2)(a_2 ^\ast,  3 )\\
		\end{split}
	\end{equation*}
	On the other hand,
	\(	\{a_0 ^\ast a_1^ \ast a_2 ^\ast, a_2 ^\ast  a_2 \} = \varepsilon(a_2 ^\ast, a_2) a_0 ^\ast a_1 ^\ast a_2 ^\ast.\)
	It is clear that the projection satisfies the  condition.
\end{example}

For comparison, we recall the standard notion of quantization for a commutative Poisson algebra.

\begin{definition}\label{def: quant Poi}
Let \(A_0\) be a commutative Poisson algebra over \(\bbC\).
A \emph{formal quantization} of \(A_0\) is an associative \(\bbC[[\hbar]]\)-algebra \(A_\hbar\) satisfying the following conditions:
\begin{enumerate}
\item \(A_\hbar\) is flat over \(\bbC[[\hbar]]\);

\item \(A_\hbar\) is complete with respect to the \(\hbar\)-adic topology;

\item there exists an isomorphism of commutative algebras
\[
A_\hbar/\hbar A_\hbar\xrightarrow{\sim}A_0;
\]
\item for any lifts \(a,b\in A_\hbar\) of \(\bar a,\bar b\in A_0 \cong A_\hbar/\hbar A_\hbar \), one has
\[
[a,b]\in\hbar A_\hbar,
\]
and the induced bracket
\begin{equation}\label{ pic}
\{\bar a,\bar b\}
:=
-\frac{1}{\hbar}[a,b]\quad\bmod\hbar
\end{equation}
coincides with the given Poisson bracket on \(A_0\).
\end{enumerate}
\end{definition}

The quantum trace realizes the representation-theoretic compatibility
predicted by the Kontsevich--Rosenberg principle.
In the present completed
and sign-adjusted convention it relates \(\qneck\) to the \(\hbar\)-Weyl
algebra \(\mD_\hbar (\Rep^{\bbC Q} _\bd)\) on \(\Rep^{\bbC Q}_\bd\).
For every coordinate function \(a_{p,q}\) on
\(\Rep^{\bbC Q}_{\mathbf d}\), let
\(\partial_{a_{p,q}}\) denote the corresponding generator of the
\(\hbar\)-Weyl algebra, i.e.
\[
\partial_{a_{p,q}}
=
\hbar\frac{\partial}{\partial a_{p,q}}.
\]
For an arrow \(a\in Q_1\), we write \([a]_{p,q} = a_{p,q}\) for its coordinate
function, and for the reverse arrow \(a^\ast \) we set
\[
[a^*]_{p,q}
:=
\partial_{a_{q,p}}.
\]

\begin{definition}[Schedler]\label{def: q trace map}
Suppose that $Q$ is a finite quiver and that $\mathbf{d}$ is a dimension vector. The \emph{quantum trace map} $\mathrm{Tr}^q$ is the continuous $\mathbb{C}[[\hbar]]$-algebra homomorphism from $\qneck$ to $\mathcal{D}_{\hbar}
	(\Rep^{\bbC Q} _\bd)$ such that, for any element of the form (\ref{general form}), its image is
	\begin{equation}
		d_{v_1}\cdots d_{v_{m}}
		\sum_{\substack{1\leq k_{i,j}\leq d_{t(a_{i,j})}\\
		                   \text{for all }i,j}}
		\left(\prod_{h=1} ^{N} [a_{\phi ^{-1}(h)}]_{k_{\phi^{-1}(h)},
			k_{\phi^{-1}(h)+1}}\right),\label{F.1}
	\end{equation}
where $\{ h_{i,j} \}=\{1,2,\ldots,N\}$ and $\phi$ is defined by $\phi(i,j)=h_{i,j}$.
\end{definition}

Here \((i,j)+1=(i,j+1)\), where the second index is taken modulo \(l_i\).
The factor \(d_{v_s}\) is the component of the dimension vector at the vertex \(v_s\), and the product in the displayed formula is ordered by increasing height.
The cited construction shows that \(\mathrm{Tr}^q\) is independent of
the chosen representative.  Multiplicativity follows directly from the
same formula: represent \(X*Y\) with all heights of \(X\) below all heights
of \(Y\); the component-index sums are independent, and the ordered
operator product places every factor from \(X\) before every factor from
\(Y\).  Hence \(\Tr^q(X*Y)=\Tr^q(X)\Tr^q(Y)\), and continuity extends the
polynomial map to the completions.  Its image is contained in
\(\mathcal D_\hbar(\Rep^{\bbC Q}_\bd)^{\gl_\bd(\bbC)}\).
The following theorem states that the quantum trace map is compatible with the two  conditions \eqref{nc  pic} and \eqref{ pic}.

\begin{theorem}\cite[Section~3.4]{Sch2005}
\label{thm: from nc  to }
Let \(Q\) be a finite quiver. Let \(\bd\) be a dimension vector.
	Then, for \(u,v\in\qneck\),
\begin{equation*}
	\{\mathrm{Tr}(\pr_0 (u)),\mathrm{Tr}( \pr_0 (v))\}  =
	-\frac1\hbar
	\left[
	\Tr^q(u),
	\Tr^q(v)
	\right]
	\bmod\hbar.
\end{equation*}
\end{theorem}

\subsection{Holland's work}\label{subsec: [Q,R]=0 on rep}

For a finite quiver \(Q\) and a dimension vector \(\bd\), define the quiver variety by
\[
\mM_\bd(Q)=\Spec\bbC[\mu^{-1}(0)]^{\GL_\bd(\bbC)}.
\]
Here \(\mu\) is the moment map on \(T^*\Rep^{\bbC Q}_\bd\):
\begin{equation*}
\mu\colon T^*\Rep^{\bbC Q}_\bd\longrightarrow\gl_\bd(\bbC)^*,
\qquad
\rho\tos\tr\!\left(\left(\sum_{a\in Q_1}[\rho_a,\rho_{a^*}]\right)\cdot-\right).
\end{equation*}

The standard quantization of \(T^*\Rep^{\bbC Q}_\bd\) is the \(\hbar\)-Weyl algebra \(\mD_\hbar(\Rep^{\bbC Q}_\bd)\). Quantum Hamiltonian reduction requires a quantum moment map.

\begin{definition}\label{def: quantum moment map}
Let \(G\) be an algebraic group with Lie algebra \(\mathfrak g\), and let \(A_\hbar\) be a flat, complete \(\mathbb C[[\hbar]]\)-algebra endowed with a \(\mathfrak g\)-action.
A continuous \(\mathbb C[[\hbar]]\)-algebra homomorphism
\[
\mu_\hbar\colon\mathcal U_\hbar^-(\mathfrak g)\longrightarrow A_\hbar
\]
is called a \emph{quantum moment map} if, for every \(v\in\mathfrak g\), the derivation
\[
A_\hbar\longrightarrow A_\hbar,
\qquad
a\tos-\frac{1}{\hbar}[\mu_\hbar(v),a],
\]
is the action of \(v\).
\end{definition}
Here, \(U_{\hbar}^{-}(\mathfrak g)\) is the completed universal enveloping algebra
\[T(\mathfrak g)[[\hbar]]\Big/\overline{
\left\langle
vw-wv+\hbar[v,w]_{\mathfrak g}
\;\middle|\;
v,w\in\mathfrak g
\right\rangle}.
\]
For quiver representation spaces, the quantum moment map is obtained from the infinitesimal \(\GL_\bd(\bbC)\)-action.
On representation points we use the standard action
\(g\cdot\rho_a=g_{t(a)}\rho_a g_{s(a)}^{-1}\), and on coordinate
functions the induced action \((g\cdot f)(\rho)=f(g^{-1}\cdot\rho)\).
The following is obtained by direct calculation; it also fixes the
elementary-matrix and transpose conventions used below.

\begin{proposition}\label{prop: chi + habr br}
Let \(Q\) be a finite quiver and \(\bd\) a dimension vector. Then the following statements hold.
\begin{enumerate}
\item The infinitesimal action of \(\GL_\bd(\bbC)\) on \({\Rep^{\bbC Q}_\bd}\) is given by the Lie algebra homomorphism
\[
\tau_{\mathrm{cl}}\colon\gl_\bd(\bbC)
\longrightarrow
\operatorname{Der}_{\bbC} \big(\bbC[\Rep^{\bbC Q}_\bd]\big)
\]
determined by
\begin{equation}\label{for: tau}
e^i_{p,q}\tos
\sum_{\substack{a\in Q_1\\s(a)=i}}\sum_{j=1}^{d_{t(a)}}[a]_{j,p}
\frac{\partial}{\partial a_{j,q}}
-
\sum_{\substack{a\in Q_1\\t(a)=i}}\sum_{j=1}^{d_{s(a)}}[a]_{q,j}
\frac{\partial}{\partial a_{p,j}}.
\end{equation}
Here \({e^i_{p,q}}\) is the elementary matrix in the \(i\)-th summand of \(\gl_\bd(\bbC)\).

\item Put
\[
\tau_\hbar(v):=\hbar\tau_{\mathrm{cl}}(v)
\in\mD_\hbar(\Rep^{\bbC Q}_\bd).
\]
Then \(\mu_\hbar(v):=-\tau_\hbar(v)\) extends to a quantum
moment map
\[
\mu_\hbar:\mathcal U_\hbar^-(\gl_\bd(\bbC))
\longrightarrow\mD_\hbar(\Rep^{\bbC Q}_\bd).
\]
For every character \(\chi:\gl_\bd(\bbC)\to\bbC\), the assignment
\(v\mapsto\mu_\hbar(v)+\hbar\chi(v)\) is also a quantum moment map.
\end{enumerate}
\end{proposition}
Note that the scalar subgroup of \(\GL_n(\bbC)\) acts
trivially, so the effective moment map is therefore
\[
\mu_{\mathrm{eff}}:
\gl_n(\bbC)\times\gl_n(\bbC)
\longrightarrow\mathfrak{pgl}_n(\bbC)^*
\cong\mathfrak{sl}_n(\bbC),
\qquad (X,Y)\longmapsto[X,Y].
\]

For a character \(\chi:\mathfrak{gl}_{\mathbf d}(\mathbb C)\to\mathbb C\), let
\[
\mathcal J_{\hbar,\chi}
:=
\overline{
\mathcal D_{\hbar}\bigl(\operatorname{Rep}^{\mathbb C Q}_{\mathbf d}\bigr)
\left\{
\tau_{\hbar}(v)-\hbar\chi(v)
\mid
v\in\mathfrak{gl}_{\mathbf d}(\mathbb C)
\right\}}
\]
be the closed left ideal generated by the shifted quantum moment-map relations. Since \(GL_{\mathbf d}(\mathbb C)\) is reductive, taking invariants commutes with the corresponding quotient modulo \(\hbar^m\) for every \(m\). Taking the inverse limit gives the same statement after \(\hbar\)-adic completion. We may therefore apply Holland’s result in the following form.

We use Holland's result in the \(\hbar\)-adically completed setting.
\begin{proposition}[{\cite[Proposition~2.4]{Hol1999}}]
\label{prop: holland}
Suppose that \(Q\) is a finite quiver, \(\bd\) is a dimension vector,
and \(\chi\) is a character of \(\gl_\bd\) that vanishes on
\(\ker\tau_{\mathrm{cl}}\).  If \(\mu_{\mathrm{eff}}\)
is flat, then
\[
\left(
\mathcal D_\hbar(\Rep^{\bbC Q}_\bd)/\mathcal J_{\hbar,\chi}
\right)^{\gl_\bd}
\]
is a formal quantization of \(\mathbb C[\mM_{\mathbf d}(Q)]\).
\end{proposition}

The completed quotient in Proposition \ref{prop: holland} is defined even
without the flatness hypothesis. We call
\[
\left(
\mathcal D_\hbar(\Rep^{\bbC Q}_\bd)/\mathcal J_{\hbar,\chi}
\right)^{\gl_\bd(\bbC)}
\]
the \emph{quantum Hamiltonian reduction} associated with \(\chi\).  We
denote this algebra by \({\mM_\bd(Q)}_{\hbar,\chi}\).

\subsection{The flatness problem}
Holland's work identifies the representation side quantum Hamiltonian reduction with a quantization of the classical quiver variety only under the stated flatness hypothesis.
For the Jordan quiver and \(n>1\), \(\mu_{\mathrm{eff}}\) is not flat.
Indeed, its zero fiber contains the locus with \(X\) regular
of dimension \(n^2+n\), whereas a flat fiber would have dimension
\[
2n^2-\dim\mathfrak{sl}_n(\bbC)=n^2+1;
\]
the zero fiber has excess dimension.
Without flatness, Holland's theorem does not establish that the representation-side quantum reduction is a quantization of the classical affine quotient.

However,
the noncommutative reduction itself requires no flatness assumption.
Section \ref{subsec: necklace} showed that preprojective algebras arise by noncommutative Hamiltonian reduction, which leads naturally to a noncommutative form of ``quantization commutes with reduction''. This was proved by one of the authors in \cite{Zhao2021Com,Zhao2023Non}.

In \cite{Zhao2021Com,Zhao2023Non},
Zhao proposed the noncommutative quantum Hamiltonian reduction.
He constructed a two-sided ideal \(\mathcal{I^{\pol}}_{\hbar}\)
and proved the quotient algebra \(\qneck^\pol / \mathcal{I^{\pol}}_{\hbar} \)
quantizes \(\Pi Q\) at the polynomial level.
Under the flatness assumption on the classical moment map and for a suitable choice of character,
Zhao (\cite{Zhao2021Com,Zhao2023Non}) proved compatibility with the Kontsevich--Rosenberg principle, expressed by the following commutative cube:
	\begin{equation}\label{com cubic}
		\begin{split}
			\xymatrixrowsep{0.8pc}
			\xymatrixcolsep{1.2pc}
			\xymatrix{
				\qneck^\pol \ar@{-->}[rd] \ar[rr]^-{\Tr^q}  &&
				\mD^\pol _{\hbar}(\Rep^{\bbC Q} _{\bd}) \ar@{-->}[rd] \\
				&\Pi Q^\pol _{\hbar, \br} \ar[rr]^{\mathrm{Tr}^{q}}
				&&
				\mM^\pol _{\bd}(Q)_{\hbar, \chi_{\br}} \\
				\QQ \ar@{~>}[uu] \ar@{-->}[rd] \ar'[r][rr]^-{\mathrm{Tr}}&&
				\bbC[T^\ast \Rep^{\bbC Q} _\bd] \ar@{~>}'[u][uu] \ar@{-->}[rd] \\
				& \Pi Q \ar@{~>}[uu] \ar[rr]^{\mathrm{Tr}}
				&& \mathbb{C}[\mM _{\mathbf{d}}(Q)] \ar@{~>}[uu]
			}	
		\end{split}
	\end{equation}

For the reader's convenience,
we now specialize to the Jordan quiver and give an \(\hbar\)-adic version of his construction.
Note that in the Jordan case, the base ring \(R = \bbC e\),
\(\QQ=A=\bbC\langle x,y\rangle\), \(\bw=xy-yx\), and
\(\Pi Q=\bbC[x,y]\).
For a  path
\(s=a_1\cdots a_m \in \QQ \),
set
\[
\widehat{[s]}:=(a_1,1)\cdots(a_m,m) \in \qneck.
\]
The path lift is extended linearly in \(s\).
In particular,
we denote the lifted vertex element by \(E\). 
Define the two-sided ideal of \(\qneck^\pol\):
\begin{equation}\label{eq:jordan-relations}
\mathcal{I}_\hbar^{\pol}:=\langle \widehat{[sxy]}-\widehat{[syx]}+\hbar \widehat{[s]}*E \mid s \in \QQ \rangle.
\end{equation}
Set
\[
\Pi Q^\pol _{\hbar} = \qneck^\pol / \mathcal{I}^\pol _\hbar,
\qquad
\Pi Q_\hbar : = \varprojlim_m
\Pi Q^\pol _{\hbar} / \hbar^m\Pi Q^\pol _{\hbar}.
\]
Note that in general \(\widehat{[s]}\ne\ell_\hbar([s])\),
since the latter depends on the PBW basis.

First,
normalize \(\widehat{[s]}\) in the fixed PBW
basis before applying \(\pr\) (see Theorem \ref{thm: schedler's}).
Skein relations contribute terms with \(\hbar\);
hence
\[
\left.\pr(\widehat{[sxy]}-\widehat{[syx]}+\hbar \widehat{[s]}*E)\right|_{\hbar=0}=[s\bw].
\]
Therefore the quotient identity gives
\[
\Pi Q_\hbar^{\pol} / \hbar \Pi Q_\hbar^{\pol}
\cong\qneck^{\pol}/
 (\mathcal I_\hbar^{\pol}+\hbar \qneck^{\pol})
\cong\Sym(\QQ _\cyc)/\langle[s\bw]\mid s \in \QQ \rangle.
\]
The kernel of \(\QQ_\cyc \to (\Pi Q)_\cyc\) is the image of
\(A\bw A\), spanned by \([p\bw]\)
since
\([a\bw b]=[ba\bw]\).
Taking symmetric algebras proves
\begin{equation}\label{eq: preproj special fibre}
	\Pi Q_\hbar^{\pol} / \hbar \Pi Q_\hbar^{\pol}
\cong\Sym((\Pi Q)_\cyc)=\Sym(\bbC[x,y]).
\end{equation}

Second, fix a dimension \(n\). The compatibility studied in \cite{Zhao2021Com,Zhao2023Non} motivates the following direct calculation with our fixed conventions. Set
\[
T:=XD-DX+\hbar nI_n,
\qquad X=(x_{ij}),\quad D=(\partial_{x_{ji}}),\quad S=s(X,D).
\]
Here \(S\) is obtained by substituting \(X,D\) into the written word \(s\) in that order. The Weyl relations and \eqref{for: tau} give
\[
T_{ij}=\sum_k\bigl(x_{ik}\partial_{x_{jk}}-x_{kj}\partial_{x_{ki}}\bigr)
=-\tau_\hbar(e_{ji}).
\]
Therefore
\begin{align*}
\Tr^q\bigl(\widehat{[sxy]}-\widehat{[syx]}+\hbar\widehat{[s]}*E\bigr)
&=\operatorname{tr}(ST)\\
&=-\sum_{i,j}S_{ij}\tau_\hbar(e_{ij}).
\end{align*}
The coefficients \(S_{ij}\) occur on the left, so this element lies in the left ideal \(\mathcal J_{\hbar,0}^{\pol}\). Every quantum trace is invariant and therefore commutes with each \(\tau_\hbar(v)\). 
Multiplicativity of \(\Tr^q\) now gives the inclusion for the entire two-sided ideal:
\begin{equation}\label{eq: qTr descends}
	\Tr^q(\mathcal I_\hbar^{\pol})
\subseteq
\mathcal J_{\hbar,0}^{\pol}
\cap\mD_\hbar^{\pol}(\gl_n)^{\GL_n},
\end{equation}
where \(\mD_\hbar^{\pol} (\gl_n (\bbC))\) denotes the uncompleted Weyl  algebra
and \(\mathcal J_{\hbar,0}^{\pol}
=\mD_\hbar^{\pol}(\gl_n)\tau_\hbar(\gl_n)\).

Main results in \cite{Zhao2021Com,Zhao2023Non} in the Jordan case are reformulated as the following theorem.
\begin{theorem}[Zhao]
\label{thm: noncom quant red for quiver}
Fix any total order on pairs \((a,b)\in\bbZ_{\geq0}^2\) with
\(a+b>0\).
The images of the ordered products
\begin{equation}\label{eq:preproj pbw basis}
E^{*m} * \widehat{[x^{a_1}y^{b_1}]} * \cdots * \widehat{[x^{a_k}y^{b_k}]},
\quad
(a_1,b_1)\leq\cdots\leq(a_k,b_k),\quad m,k\geq0,
\end{equation}
form a \(\bbC[\hbar]\)-basis of \(\Pi Q_\hbar^{\pol}\).
In particular,
\(\Pi Q_\hbar\) is a noncommutative quantization of
\(\Pi Q\).
\end{theorem}
For the reader's convenience,
we give a self-contained proof of the theorem.
The following is a proposition needed in the proof.
\begin{proposition}\label{prop: joint inj}
Let
\(\Com_n\) be the affine
scheme defined by commuting \(n \times n\)-matrices \((X,Y)\).
Let
\[
\Tr:
\Sym(\bbC[x,y]) \to
\bbC [\Com_n // \GL_n]
\]
be the classical trace map.
If \(S\in \Sym(\bbC [x,y])\) satisfies
\(\Tr(S)=0\) for all sufficiently large
integers \(n\), then \(S=0\).
\end{proposition}

\begin{proof}
For \(a,b\geq0\), write \(z_{a,b}=x^ay^b\), viewed as a
generator of the symmetric algebra. Thus
\[
\operatorname{Sym}(\mathbb C[x,y])
=\mathbb C[z_{a,b}]_{a,b\geq0}.
\]

Suppose, for a contradiction, that \(S \neq 0\), and choose
\(n_0 \geq 1\) such that \(\Tr( S)=0\) whenever
\(n\geq n_0\).
By definition,
one can write
\[
S=\sum_{\lambda\in\Lambda}c_\lambda(z_{0,0})
\prod_{(a,b)\in\lambda}z_{a,b}.
\]

Here \(\Lambda\) is a finite set of distinct finite multisets of pairs
\((a,b)\in\bbZ_{\geq0}^2\setminus\{(0,0)\}\), and each
\(c_\lambda(t)\in\mathbb C[t]\) is nonzero. 
Choose any \(\lambda_*\in\Lambda\) with
\[
k:=|\lambda_*|=\max_{\lambda\in\Lambda}|\lambda|.
\]
Here \(|\lambda|\) counts indices in \(\lambda\) itself.
We set
\(\lambda_*\) as \(\{(a_1,b_1),\ldots,(a_k,b_k)\}\).
Since \(c_{\lambda_*}\) is a nonzero polynomial, there exists
an integer
\begin{equation}\label{assp n}
	n\geq\max\{n_0,k,1\}
\quad\text{such that}\quad c_{\lambda_*}(n)\neq0.
\end{equation}
If \(k=0\), then \(S=c_{\lambda_*}(z_{0,0})\), so
\(\Tr(S)=c_{\lambda_*}(n)\neq0\), a contradiction.
Hence we may assume that \(k\geq1\).

Let \(u_1,\ldots,u_n,v_1,\ldots,v_n\) be independent
indeterminates. The diagonal matrices
\[
X=\operatorname{diag}(u_1,\ldots,u_n),
\qquad
Y=\operatorname{diag}(v_1,\ldots,v_n)
\]
commute, and therefore define a morphism
\(\bbC^n \times \bbC^n \to \Com_n \to \Com_n //\GL_n\).
Pulling back the identity \(\Tr(S)=0\)
along this morphism gives
\begin{equation}\label{for: Tr S on diag}
	0=\sum_{\lambda\in\Lambda}c_\lambda(n)
\prod_{(a,b)\in\lambda}
\left(\sum_{i=1}^n u_i^av_i^b\right)
\end{equation}
in \(\mathbb C[u_1,\ldots,u_n,v_1,\ldots,v_n]\).

For a monomial in the variables \(u_i,v_i\), call an index
\(i\) occupied if the exponent of \(u_i\) or of \(v_i\)
is positive. 
For example, in the polynomial
\[
\left(\sum_{i=1}^n u_i^2\right)
\left(\sum_{j=1}^n v_j\right)
=\sum_{i=1}^n u_i^2v_i
 +\sum_{\substack{1\leq i,j\leq n\\ i\neq j}}u_i^2v_j,
\]
\(u_1^2v_1\) occupies only index \(1\); whereas \(u_1^2v_2\)
occupies indices \(1\) and \(2\).

Our fixed \(\lambda_\ast = \{(a_1,b_1),\ldots,(a_k,b_k)\}\) contributes a special monomial
\[
m_*=\prod_{r=1}^{k}u_r^{a_r}v_r^{b_r}.
\]
Since \(a_r+b_r>0\) for every \(r\),
the occupied indices of \(m_*\) are exactly \(1,\ldots,k\).
Since an expanded term arising from the summand indexed by
\(\lambda\) occupies at most \(|\lambda|\) indices,
only summands with \(|\lambda|=k\) can contribute to
\(m_*\), and each such contribution must assign the \(k\)
factors to distinct indices.
Its multiset of exponent pairs must consequently equal
\(\{(a_1,b_1),\ldots,(a_k,b_k)\}\), including multiplicities.
This forces \(\lambda=\lambda_*\).

For each exponent pair \(\alpha\) occurring in \(\lambda_\ast\),
let \(m_\alpha\) be its multiplicity.
The coefficient of
\(m_*\) in the displayed polynomial \eqref{for: Tr S on diag} is therefore
\[
c_{\lambda_*}(n)\prod_\alpha m_\alpha!,
\]
which is nonzero due to \eqref{assp n}. This contradicts the polynomial identity
above. Hence \(S=0\).
\end{proof}
The main idea of the proof is that, for \(n\) chosen as in \eqref{assp n}, the restriction of \( \Tr(S) \) to diagonal matrices must vanish, whereas \(n\geq k\) and \(c_{\lambda_*}(n)\ne0\) produce a monomial with nonzero coefficient.

\begin{proof}[Proof of Theorem \ref{thm: noncom quant red for quiver}]
We prove the polynomial basis assertion first and then verify the formal quantization statement after completion.

Recall that
\eqref{eq: preproj special fibre} gives
\[
E\bmod\hbar=z_{0,0},
\qquad
\widehat{[x^ay^b]}\bmod\hbar=z_{a,b}.
\]
Note that \(z_{0,0}\) is a generator of the symmetric
algebra \(\Sym ((\Pi Q)_{\mathrm{cyc}})
 =\Sym(\mathbb C[x,y])\),
distinct from its scalar unit.

Let
\[
\{M_\lambda = E^{*m}*\widehat{[x^{a_1}y^{b_1}]}*\cdots*\widehat{[x^{a_k}y^{b_k}]}\ \big\vert\ \lambda = \big( m , (a_i, b_i) \big) \}
\]
be the family of ordered products
in the statement, viewed in \(\Pi Q^\pol _\hbar \).
Modulo \(\hbar\),
their residues
\(\overline M_\lambda\) are precisely the monomials
\[
z_{0,0}^{\,m}z_{a_1,b_1}\cdots z_{a_k,b_k},
\qquad
(a_1,b_1)\leq\cdots\leq(a_k,b_k),\qquad m,k\geq0,
\]
with \(a_i+b_i>0\). Hence they form a \(\mathbb C\)-basis of \(\Sym(\mathbb C[x,y])\).

Next, we prove that the \(M_\lambda\) span \(\Pi Q^\pol _\hbar\)
over
\(\mathbb C[\hbar]\).
On the algebra \(\qneck^\pol\),
we consider the grading in which each arrow occurrence has degree
\(1\), each vertex factor has degree \(0\), and \(\hbar\) has
degree \(2\).
Both skein relations are homogeneous, since
each contraction removes two arrows.
The defining relation
\[
\widehat{[sxy]}-\widehat{[syx]}+\hbar \widehat{[s]}*E
\]
is homogeneous of degree \(|s|+2\),
so this grading descends
to a nonnegative grading \(\Pi Q^\pol _\hbar = \bigoplus_{d\geq0} (\Pi Q^\pol _\hbar)_d \).
The semiclassical limit respects the induced
grading on \(\Sym(\mathbb C[x,y])\), where \(\deg z_{a,b}=a+b\).

Let \(u \in (\Pi Q^\pol _\hbar)_d\).
Since the \(\overline M_\lambda\) form a
homogeneous basis of \(\Sym(\bbC[x,y])\),
there is a finite
\(\bbC\)-linear combination \(L_0\) of the \(M_\lambda\)
of degree \(d\) such that \(u-L_0\in\hbar \Pi Q_\hbar ^\pol\).
Repeating this argument gives
\[
u=\sum_{r=0}^{\lfloor d/2\rfloor}\hbar^rL_r,
\]
where each \(L_r\) is a finite \(\mathbb C\)-linear
combination of the \(M_\lambda\) of degree \(d-2r\).
The process terminates because the grading is nonnegative.
Since every element of \(\Pi Q_\hbar^\pol\) is a finite sum of homogeneous
elements, the required spanning assertion follows.

We prove linear independence.
In Section \ref{subsec: tilde Tr},
we introduce a radial quantum trace map and its polynomial version is written as
\[
\widetilde{\Tr^q}^\pol:\Pi Q_\hbar^{\pol}
\longrightarrow \mD_\hbar^{\pol}(\fkh)^{S_n}.
\qquad n\geq1.
\]
Note that existence of this morphism is independent of the present theorem.

Suppose
\[
\sum_{j=1}^k f_j(\hbar)M_{\lambda_j}=0
\]
is a nontrivial relation with \(f_j\in\mathbb C[\hbar]\) and pairwise
distinct \(\lambda_j\).
Let \(r\ge 0\) be the largest integer such that \(\hbar^r\) divides
every \(f_j\), and write \(f_j(\hbar)=\hbar^r g_j(\hbar)\) with
\(g_j\in\mathbb C[\hbar]\).
By construction, at least one \(g_j(0)\neq 0\).
For each \(n\), apply
\(\widetilde{\Tr^q}^\pol\) to the relation before cancelling any factor.
The target \(\mD_\hbar^\pol(\fkh)^{S_n}\) is
\(\hbar\)-torsion-free, since it is a submodule of the polynomial Weyl
algebra. We can therefore cancel \(\hbar^r\) in the target and obtain
\[
\sum_{j=1}^k g_j(\hbar)
\widetilde{\Tr^q}^\pol(M_{\lambda_j})=0.
\]
The polynomial construction in Section~\ref{subsec: tilde Tr} specializes
on \(E\) to \(n\) and on \(\widehat{[x^ay^b]}\) to
\(\sum_{i=1}^n x_i^ay_i^b\), independently of the present PBW assertion.
Consequently
\[
S:=\sum_{j=1}^k g_j(0)\overline M_{\lambda_j}
\]
has zero diagonal trace for every \(n\). By the diagonal-trace argument in the proof of
Proposition~\ref{prop: joint inj}, \(S=0\).
The linear independence of the \(\overline M_{\lambda_j}\) implies
\(g_j(0)=0\) for all \(j\), a contradiction.
This proves the basis assertion.

Finally, completion identifies \(\Pi Q_\hbar\) with
\(\Sym(\mathbb C[x,y])[[\hbar]]\) as a module,
hence it is flat and complete.
Together with the special-fiber identification
\eqref{eq: preproj special fibre} and the negative  identity
descending from \eqref{nc  pic} through the cyclic Lie quotient
of Proposition \ref{prop: noncom red},
this proves the quantization assertion.
\end{proof}

Now,
Theorem \ref{thm: noncom quant red for quiver} is summarized by the schematic commutative diagram
\begin{equation}
\begin{split}
\xymatrixcolsep{4pc}
\xymatrix{
	\qneck \ar@{-->}[r]
& {\Pi Q}_{\hbar} \\
\QQ \ar@{~>}[u] \ar@{-->}[r]& \Pi Q.
\ar@{~>}[u]}
\end{split}
\end{equation}
The horizontal arrows represent Hamiltonian reduction, while the vertical arrows represent the relevant quantizations.
Moreover, the quantum-trace-descent property  remains valid without flatness (see \eqref{eq: qTr descends}). Hence,
in the Jordan case and for any positive integer \(n\),
one always obtains a homomorphism
\[
{
\Pi Q_{\hbar} \longrightarrow \mM_n (Q)_{\hbar,0},
}
\]
although the target is not known from Holland's work to be a quantization of \(\bbC[\mM_n (Q)]\) when \(\mu_{\rm eff}\) is non-flat.

Meanwhile,
restriction to diagonal
commuting pairs gives an isomorphism of affine quotients
\[
\mM_n (Q) = \Com_n//\GL_n
\cong
(\bbC^n\oplus\bbC^n)//S_n
\cong
\rmT^*\bbC^n//S_n.
\]
We suggest \cite{GanGin2006,Eti2009Lec} for further details and references.
This suggests replacing the unavailable flatness argument with a direct comparison between \(\Pi Q_{\hbar}\) and the invariant differential operators on \(\bbC^n\).
The next section constructs this comparison.

\section{The quantum homomorphism}\label{sec: quant mor}

This section presents the principal construction of the paper and develops the explicit formulas needed later.  
We first construct the radial quantum trace morphism
\[
\widetilde{\Tr^q}:
\Pi Q_{\hbar}
\longrightarrow
\mD_\hbar(\bbC^n)^{S_n},
\]
and then compute its radial and Harish--Chandra images explicitly.

\subsection{Construction of \texorpdfstring{\({\widetilde{\Tr^q}}\)}{the radial quantum trace}}\label{subsec: tilde Tr}

We begin by recalling the radial-part and Harish--Chandra homomorphisms, which are the representation-theoretic input for the construction of \(\widetilde{\Tr^q}\).

Let \(\rmG\) be a connected complex reductive group with Lie algebra
\(\fkg\), Cartan subalgebra \(\fkh\), root system
\(R=R(\fkg,\fkh)\), and Weyl group \(W\).
Choose positive roots
\(R^+\subset R\) and set
\[
\fkh_{\reg}
:=
\fkh\setminus\bigcup_{\alpha\in R}\ker(\alpha),
\qquad
\delta:=\prod_{\alpha\in R^+}\alpha.
\]
The radial-part homomorphism
\[
\hc':\mD(\fkg)^{\fkg}\longrightarrow\mD(\fkh_{\reg})^W
\]
is characterized by
\[
\hc'(P)(f|_{\fkh})=P(f)|_{\fkh},
\qquad f\in\bbC[\fkg]^{\fkg},
\]
and the Harish--Chandra homomorphism is
\[
\hc(P)=\delta\,\hc'(P)\,\delta^{-1}.
\]

Let \(\tau_{\mathrm{ad}}:\fkg\to\mD(\fkg)\) denote the infinitesimal
adjoint action, and write
\[
\mD(\fkg)///\fkg
:=
\left(
\mD(\fkg)\big/\mD(\fkg)\tau_{\mathrm{ad}}(\fkg)
\right)^{\fkg}.
\]
Fundamental properties are summarized as follows.
\begin{theorem}\label{thm: hc regular}
Let \(\rmG\) be a connected complex reductive group with Lie algebra
\(\fkg\), Cartan subalgebra \(\fkh\), and Weyl group \(W\).
\begin{enumerate}
\item \textup{(Harish--Chandra \cite{HC1964Inv})} The map
\[
\hc:\mD(\fkg)^\fkg \to
\mD(\fkh_{\reg})^W
\]
extends the Chevalley restriction isomorphisms
\[
\bbC[\fkg]^\fkg \cong\bbC[\fkh]^W,
\qquad
\Sym(\fkg)^\fkg \cong \Sym(\fkh)^W,
\]
and its image is contained in \(\mD(\fkh)^W\).

\item \textup{(Levasseur--Stafford \cite{LevSta1995Inv,LevSta1996Ker})} 
The map \(\hc\) induces an 
isomorphism
\[
\mD(\fkg)///\fkg\xrightarrow{\sim}\mD(\fkh)^W.
\]
\end{enumerate}
\end{theorem}

{We shall use the preceding result in the \(\hbar\)-adic setting.
The Harish--Chandra homomorphism \(\hc\)
is filtered with respect to the order filtrations on 
differential operators.
Hence it induces a homomorphism \(\hc^{\rm Weyl}\) between the corresponding Weyl algebras.
Since this \(\hc^{\rm Weyl}\) preserves
the powers of \(\hbar\), it is continuous for the \(\hbar\)-adic topology and therefore extends uniquely to the \(\hbar\)-adic completions \(\hc_\hbar\).
By abuse of notation,
we still denote this extension by \(\hc\).}

{Moreover, by Theorem \ref{thm: hc regular}, the Harish--Chandra homomorphism factors through the adjoint quantum
Hamiltonian reduction.  Applying the Weyl algebra construction and then
\(\hbar\)-adic completion therefore gives a continuous homomorphism
\[
\hc:
\mD_\hbar(\fkg)///\fkg
\longrightarrow
\mD_\hbar(\fkh)^W .
\]
We use this completed form throughout the sequel.}

The following fact is needed in later sections, see \cite[Lemma 4.6]{Eti2009Lec} for calculation.
\begin{proposition}\label{prop: hc laplacian}
Let \(\fg\) be a reductive Lie algebra and let
\(\fkh\subset\fg\) be a Cartan subalgebra. Fix a nondegenerate
invariant symmetric bilinear form on \(\fg\), and let
\(\Delta_{\fg}\) and \(\Delta_{\fkh}\) be the Laplacians defined by
this form and its restriction to \(\fkh\). Then
\[
\hc(\Delta_{\fg})=\Delta_{\fkh}.
\]
\end{proposition}

\begin{remarks}\label{rmk: comput laplace}
The Laplacian case shows that taking the radial part of an invariant differential operator is not equivalent to discarding its off-diagonal terms: derivatives in the root directions contribute nontrivially before Harish--Chandra conjugation.
\end{remarks}

A key property of the Harish--Chandra homomorphism is its compatibility with principal symbols restriction.
Consider the order filtration \(F_\bullet \mD(\fkg)\) on $\mD(\fkg)$, and let \(\sigma_\bullet: \mD(\fkg)\to\Sym^\bullet (\fkg\oplus\fkg^*)\) denote the symbol map.  If 
\(P\in F_m\mD(\fkg)^{\rmG}\), then
\[
\sigma_m(\hc(P))
=
\left.\sigma_m(P)\right|_{\fkh\oplus\fkh^*}.
\]
Indeed, the leading term of radial restriction is obtained by restricting
both the base and cotangent variables, while conjugation by the
multiplication operator \(\delta\) changes only terms of order strictly
less than \(m\).  Consequently,
\[
\gr\hc:
\bbC[T^*\fkg]^{\fkg}
\longrightarrow
\bbC[T^*\fkh]^W
\]
is the restriction map, and on completed Weyl algebras
\(\hc_\hbar\bmod\hbar=\operatorname{res}\).  After adjoint Hamiltonian
reduction, this compatibility is given by \cite[Theorem~6.10(ii)]{EG2002}.

We now specialize to the Jordan quiver.  Then
\(\rmG=\GL_n(\bbC)\), \(\fkg=\gl_n(\bbC)\),
\(\fkh\simeq\bbC^n\), and \(W=S_n\).
Composing the descended quantum trace with the completed
Harish--Chandra homomorphism gives
\[
\widetilde{\Tr^q}:
\Pi Q_{\hbar}
\longrightarrow
\mD_\hbar (\gl_n (\bbC))///\gl_n(\bbC)
\xrightarrow{\ \hc\ }
\mD_\hbar(\bbC^n)^{S_n}.
\]

It is natural to expect that the specialization at \(\hbar=0\) should recover the classical trace morphism from the noncommutative reduction to the Cartan quotient.  The next proposition makes this statement precise.

\begin{proposition}\label{prop: tilt Tr semiclassical}
{Let \(Q\) be the Jordan quiver and let \(n\) be a positive integer.  Reduction modulo \(\hbar\) of \(\widetilde{\Tr^q}\) defines a Poisson algebra morphism
\[
\tTr:
\Sym((\Pi Q)_\cyc)
\longrightarrow
\bbC[\rmT^\ast\bbC^n]^{S_n}.
\]
For any \(x,y\in(\Pi Q)_\cyc\) and lifts \(\widehat{x},\widehat{y}\in\Pi Q_{\hbar}\), one has
\[
\{\tTr(x),\tTr(y)\}
=
-\frac{1}{\hbar}
\bigl[
\widetilde{\Tr^q}(\widehat{x}),
\widetilde{\Tr^q}(\widehat{y})
\bigr]
\quad\bmod\hbar.
\]}
\end{proposition}

\begin{proof}
By Section \ref{subsec: nc quant} and Theorem \ref{thm: noncom quant red for quiver}, the source specializes as
\[
\Pi Q_{\hbar}/\hbar\Pi Q_{\hbar}
\cong
\Sym((\Pi Q)_\cyc).
\]
The target specializes as
\[
\mD_\hbar(\fkh)^{S_n}/\hbar\mD_\hbar(\fkh)^{S_n}
\cong
\bbC[\rmT^\ast\fkh]^{S_n}.
\]
Here invariants commute with reduction modulo \(\hbar\), because the
Reynolds operator \(\frac1{n!}\sum_{\sigma\in S_n}\sigma\) makes the
\(S_n\)-invariants functor exact.
Modulo \(\hbar\), Schedler's work \cite{Sch2005} and Zhao's work \cite{Zhao2023Non} show that the quantum trace map becomes the classical trace map,
which is an algebra morphism with respect to the symmetric product on
\(\Sym((\Pi Q)_\cyc)\).

Analysis below Remark \ref{rmk: comput laplace} shows that specialization of \(\hc\) is the restriction of the corresponding principal symbol from \(\rmT^\ast\gl_n\) to \(\rmT^\ast\fkh\).
Thus the specialization of \(\widetilde{\Tr^q}\) is an algebra morphism \(\tTr\).

Both source and target carry their semiclassical Poisson brackets
through \(-\hbar^{-1}\) times the commutator. Since
\(\widetilde{\Tr^q}\) is an algebra morphism, it preserves
commutators, and the stated identity follows after reduction modulo
\(\hbar\). Hence \(\widetilde{\Tr}\) is Poisson.
\end{proof}

\begin{proposition}\label{prop: tTr surj}
Let \(Q\) be the Jordan quiver and let \(n\) be a positive integer.
Then the morphisms
	\(\widetilde{\Tr^q},\ \widetilde{\Tr}\) are surjective algebra morphisms.
\end{proposition}

\begin{proof}
We first prove the surjectivity of \(\widetilde{\Tr}\).  
For the Jordan
quiver, we have
		\(\Pi Q\cong \bbC[x,y] \cong (\Pi Q)_\cyc\).
	Recall that we identify
	\(\fkh\simeq \bbC^n,\ 
		\rmT^\ast\fkh\simeq \fkh\oplus\fkh^\ast\).
	Let
	\(x_1,\dots,x_n\)
	be the standard coordinate functions on \(\fkh\), and let
	\(y_1,\dots,y_n\)
	be the corresponding coordinate functions on \(\fkh^\ast\).
	The Weyl group
	\(S_n\) acts diagonally on \(\rmT^\ast\fkh\simeq \fkh\oplus\fkh^\ast\).
	
	By the definition of the classical trace map, one has
	\[
	\widetilde{\Tr}([x^ay^b])
	=
	p_{a,b}
	=
	\sum_{i=1}^n x_i^ay_i^b
	\in
	\bbC[\rmT^\ast\fkh]^{S_n},\qquad a,b \in \bbZ_{\geq 0}.
	\]
	The functions \(p_{a,b}\) are the polarized power sums for the diagonal
	\(S_n\)-action on \(\rmT^\ast\fkh\).
	A standard result in invariant theory shows that
	the invariant ring
	\[
	\bbC[x_1,\dots,x_n,y_1,\dots,y_n]^{S_n}
	\]
	is generated as a \(\bbC\)-algebra by the polarized power sums
	\[
	p_{a,b}
	=
	\sum_{i=1}^n x_i^ay_i^b,
	\qquad a, b \in \bbZ_{\geq 0}\text{ and }
	1\leq a+b\leq n.
	\]
	See \cite[Theorem 2.5]{Dom2009Vec} and the references therein.
	Since each generator \(p_{a,b}\) lies in the image of
	\(\widetilde{\Tr}\), the morphism
	\(\widetilde{\Tr}:
	\Sym((\Pi Q)_\cyc)
	\longrightarrow
	\bbC[\rmT^\ast\fkh]^{S_n}\)
	is surjective.
	
	It remains to lift surjectivity to the quantum morphism
	\(\widetilde{\Tr^q}:
	\Pi Q_{\hbar}
	\longrightarrow
		\mD_\hbar(\fkh)^{S_n}\).
	Both algebras are complete for the \(\hbar\)-adic
	topology: this is part of the quantization statement.  The map \(\widetilde{\Tr^q}\) is continuous and
	its reduction modulo \(\hbar\) is the surjective map
	\(\widetilde{\Tr}\) proved above.

	Let \(\beta \in\mD_\hbar(\fkh)^{S_n}\).
	Choose \(\alpha_0\in\Pi Q_{\hbar}\)
	whose image under \(\widetilde{\Tr^q}\) agrees with \(\beta\) modulo \(\hbar\).  Inductively, suppose
	\(\alpha_0,\ldots,\alpha_{k-1}\) have been chosen so that
	\[
	\beta - \widetilde{\Tr^q}\!\left(\sum_{j=0}^{k-1}\hbar^j\alpha_j \right)
	\in \hbar^k\mD_\hbar(\fkh)^{S_n}.
	\]
	Write this difference as \(\hbar^k \beta_k\).  Surjectivity modulo \(\hbar\)
	provides \(\alpha_k \in \Pi Q_{\hbar}\) with
	\(\widetilde{\Tr^q}(\alpha_k) \equiv \beta_k \pmod \hbar \).
	Completeness of
	\(\Pi Q_{\hbar}\) gives
	\(\alpha=\sum_{k\geq0}\hbar^k\alpha_k \in \Pi Q_{\hbar}\), and continuity gives
	\(\widetilde{\Tr^q}(\alpha)=\beta\).  Thus \(\widetilde{\Tr^q}\) is surjective.
\end{proof}

The preceding results can be summarized as follows.

\begin{theorem} \label{thm:intro-radial-qtrace}
Let \(Q\) be the Jordan quiver and let \(n\geq1\).
\begin{enumerate}
\item[$(1)$] There is a surjective \(\bbC[[\hbar]]\)-algebra homomorphism
\[
\widetilde{\Tr^q}:
\Pi Q_{\hbar}
\longrightarrow
\mD_\hbar(\fkh)^{S_n}.
\]

\item[$(2)$] The semiclassical limit of \(\widetilde{\Tr^q}\),
\[
\widetilde{\Tr}:
\Sym((\Pi Q)_\cyc)
\longrightarrow
\bbC[T^*\fkh]^{S_n}
\]
is a surjective {Poisson} algebra morphism.  More precisely,
\[
\widetilde{\Tr}([x^ay^b])
=
\sum_{i=1}^n x_i^ay_i^b.
\]
\end{enumerate}
\end{theorem}

Although the construction of \(\widetilde{\Tr^q}\) is formal once the descent statement is established, computing the image of a mixed invariant differential operator is substantially more difficult.  To indicate the source of this difficulty, let
\[
I
:=
\left\{
D\in \mD(\fkg)^{\rmG}
\;\middle|\;
D(f)=0
\text{ for every }
f\in \bbC[\fkg]^{\rmG}
\right\}.
\]
The work of Harish--Chandra \cite{HC1957Dif, HC1964Inv}, together with that of Wallach \cite{Wal1993Inv} and Levasseur--Stafford \cite{LevSta1995Inv,LevSta1996Ker}, gives the exact sequence
\[
0
\longrightarrow I
\longrightarrow \mD(\fkg)^{\rmG}
\xrightarrow{\hc}
\mD(\fkh)^{W}
\longrightarrow0.
\]
Moreover,
\[
\hc(f)=\left.f\right|_{\fkh},
\qquad
f\in\bbC[\fkg]^{\rmG},
\]
and
\[
\hc\bigl(P(\partial)\bigr)
=
\left.P\right|_{\fkh}(\partial),
\qquad
P\in\operatorname{Sym}(\fkg)^{\rmG}.
\]

Levasseur--Stafford \cite{LevSta1995Inv} proved that
\(\mD(\fkh)^{W}\) is generated by the two subalgebras
\[
\bbC[\fkh]^{W}
\qquad\text{and}\qquad
\operatorname{Sym}(\fkh)^{W},
\]
acting respectively by multiplication and by constant-coefficient differential operators.  Consequently, once an invariant operator is represented modulo \(I\) as a noncommutative word in the corresponding source subalgebras, its Harish--Chandra image is obtained by substitution.  The substantive difficulty is that naturally occurring mixed differential operators are rarely presented in this form; finding such a representative modulo \(I\) is itself a nontrivial invariant-theoretic problem.

\subsection{Computation of the radial-part homomorphism}
\label{subsec: state radial part in powersum}

To describe \(\widetilde{\Tr^q}\) explicitly, we compute the radial parts of the
invariant differential operators arising from quantum traces.
These quantum trace operators are the basic building blocks for the image of the quantum trace map.
By Schedler's construction in Section \ref{subsec: nc quant} and the similar lifting argument in the proof of Proposition \ref{prop: tTr surj},
every element of
\(\mD_\hbar (\gl_n (\bbC))^{\gl_n (\bbC)}\) is an
\(\hbar\)-adic limit of finite sums of products of such operators.

For \(s,r \geq 0\), let \(w_{s,r}\in\qneck\) be the
single heighted cyclic path
\[
w_{s,r}
=
(x,1)\cdots(x,s)(y,s+1)\cdots(y,s+r) \in \qneck.
\]
Schedler's formula \cite[Eq.~(3.12)]{Sch2005}, in the convention of
Definition~\ref{def: q trace map}, gives
\begin{equation}\label{for: trace w_sr}
	\Tr^q(w_{s,r})
	=
	\sum_{q_0,q_1,\ldots,q_r=1}^n
	(X^s)_{q_0,q_1}
	\prod_{j=1}^r\partial_{x_{q_{j+1},q_j}}
	\in\mD_\hbar(\gl_n (\bbC))^{\gl_n (\bbC)}.
\end{equation}
Here \(q_{r+1}:=q_0\).
To give a compact formula, we recall that
\(
	D=(\partial_{x_{ji}})
\)
where \(\partial_{x_{ij}}\) denotes the
\(\hbar\)-Weyl algebra generator; on polynomial functions it acts as
\[
\partial_{x_{ij}}
=
\hbar\frac{\partial}{\partial x_{ij}}.
\]

With the matrix multiplication convention, the preceding \(\Tr^q(w_{s,r})\) can be written as the trace of the matrix \(X^s D^r\).
By abuse of notation we denote it by
\(\Tr^q(X^sD^r)\).  
Write \(\overline w_{s,r}\) for the image of \(w_{s,r}\) in
\(\Pi Q_{\hbar}\).

We now introduce the notation needed for the radial-part formula for
\(\Tr^q(X^sD^r)\).
Let
\[
q=(q_0,q_1,\ldots,q_r),\qquad q_j\in\{1,\ldots,n\},
\]
and set \(q_{r+1}=q_0\). Put
\[
M_q(u)
=
X+\sum_{j=1}^{r}u_jE_{q_{j+1},q_j}
=
X+\sum_{j=2}^{r+1}u_{j-1}E_{q_j,q_{j-1}},
\]
where \(u_1,\ldots,u_r\) are commuting indeterminates and
\(E_{a,b}\) denotes the matrix unit with a \(1\) in position
\((a,b)\).

No distinctness condition is imposed on the indices \(q_j\), and throughout this section \(\sum_q\) means \(\sum_{q_0,\dots,q_r=1}^n\).
For \(k\geq0\), define \(\Gamma_{q,I}^{(k)}(X)\) as the coefficient of \(u^I\)
in \(M_q(u)^k\):
\[
	M_q(u)^k
	=
	\sum_{I\in\bbZ_{\geq0}^r}
	\Gamma_{q,I}^{(k)}(X)u^I.
\]
Here, \(u^I=u_1^{\rho_1}\cdots u_r^{\rho_r}\) for
\(I=(\rho_1,\dots,\rho_r)\in\bbZ_{\geq 0}^r\).
If \(B\subseteq [r]=\{1,\dots,r\}\), let \(I_B\in\{0,1\}^r\) be its
indicator multi-index.
For example, if \(B=\{3,5,7\}\subset [7]\), then \(I_B=(0,0,1,0,1,0,1)\). 
A straightforward calculation shows that
\[
\left.
\left(\prod_{j\in B}\frac{\partial}{\partial u_j}\right)
M_q(u)^k
\right|_{u=0}
=
\Gamma_{q,I_B}^{(k)}(X).
\]

For every \(k\in\bbZ_{\geq0}\), set
\[
p_k:=\sum_{i=1}^n x_i^k;
\qquad\text{in particular, }p_0=n.
\]
The Newton identities give
\(\bbC[\fkh]^{S_n}=\bbC[p_1,\ldots,p_n]\).  Since
\(\fkh_{\reg}\to\fkh_{\reg}/S_n\) is finite \'{e}tale, each coordinate
derivation \(\hbar\,\partial/\partial p_a\) has a unique
\(S_n\)-invariant lift to \(\fkh_{\reg}\).  We denote it by
\(\partial_{p_a}\).  Then, for \(1\leq a,b\leq n\),
\[
[\partial_{p_a},p_b]=\hbar\delta_{a,b},
\qquad
[\partial_{p_a},\partial_{p_b}]=0.
\]
Recall that in this work, the symbols
\(\frac{\partial}{\partial u_j}\) denote ordinary derivations,
whereas \(\partial_{p_a}\) denotes the element of the \(\hbar\)-Weyl algebra.
For a set partition \(\pi\) of \([r]\) (see Appendix \ref{app: fdb} for the definition) and a map 
\(\alpha:\pi\to\{1,\dots,n\}\), set
\[
	\Tr(\Gamma_{q,\pi}^{\alpha})
	:=
	\prod_{B\in\pi}
	\Tr\!\left(\Gamma_{q,I_B}^{(\alpha(B))}\right),
	\quad
		\partial_p ^\alpha
	:=
	\prod_{B\in\pi}
	\partial_{p_{\alpha(B)}},
	\quad
	\sum_{(\pi,\alpha)}
	=
	\sum_{\pi\in\Pi([r])}
		\sum_{\alpha:\pi\to\{1,\dots,n\}}.
	\]
The product defining \(\partial_p^\alpha\) is independent of the
ordering of the blocks because the lifted derivations commute.
With this notation, we have

\begin{theorem}[Radial-part homomorphism in power sums]
\label{thm: radial part}
Let \(\fkg=\gl_n(\bbC)\) and let \(\fkh\) be the 
Cartan subalgebra of diagonal matrices.
Then for any \(s\in\bbZ_{\geq 0}\) and \(r \in \bbZ_{\geq 1}\),
\begin{equation*}
\hc'(\Tr^q(X^sD^r))
=
		\sum_{(\pi,\alpha)}
		\hbar^{r - |\pi|}
		\left(
		\sum_{q}
		{\delta_{q_0,q_1}} x_{q_0}^s
		\Tr(\Gamma_{q,\pi}^{\alpha}(x_1,\dots,x_n))
		\right)
		\pa^\alpha_p.
\end{equation*}
Furthermore, each coefficient
\[
\sum_q {\delta_{q_0, q_1}} x_{q_0}^s
\Tr(\Gamma_{q,\pi}^{\alpha}(x_1,\dots,x_n))
\]
is a polynomial in the power sums \(p_k(x_1,\dots,x_n)\).
\end{theorem}

The degenerate case \(r=0\) is handled by Theorem \ref{thm: hc regular}.
 The expanded coefficient formula appears in \eqref{for: radial part coeff}.

\subsection{Proof of Theorem \ref{thm: radial part}}

We first make the coefficients 
\(\Tr(\Gamma_{q,I}^{(k)}(X))\) explicit.
Let \(n,r\geq 1\), and fix indices
\[
q=(q_0,q_1,\dots,q_r),
\qquad
q_j\in\{1,\dots,n\}.
\]
No distinctness condition is imposed on the 
indices \(q_j\).

For \(I=(\rho_1,\dots,\rho_r)\neq 0\) in \(\mathbb Z_{\geq 0}^{r}\), set
\(m=|I|:=\sum_{j=1}^r\rho_j\), and define
\[
\Ord(I)
:=
\Biggl\{
	(i_1,\dots,i_m)\in\{1,\dots,r\}^m
	\ \bigg\vert\
	\#\{a\in\{1,\dots,m\}\mid i_a=j\}=\rho_j
	\text{ for every }j
\Biggr\}.
\]
Equivalently, \(\Ord(I)\) is the set of all orderings of the multiset
containing \(\rho_j\) copies of the symbol \(j\).
Expanding \(M_q(u)^k\) gives the following formula.

\begin{lemma}
Let \(k\in\bbZ_{\geq0}\) and
\(I=(\rho_1,\ldots,\rho_r)\in\mathbb Z_{\geq 0}^{r}\), and set
\(m=|I|\).
\begin{enumerate}
\item If \(m>k\), then
\[
\Gamma_{q,I}^{(k)}(X)=0.
\]
\item If \(I=(0,\ldots,0)\), then
\[
\Gamma_{q,0}^{(k)}(X)=X^k.
\]
\item If \(1\leq m\leq k\), then
\[
\Gamma_{q,I}^{(k)}(X)
=
\sum_{i=(i_1,\ldots,i_m)\in \Ord (I)}
\ \sum_{\substack{\ell_0,\ldots,\ell_m\geq 0\\
\ell_0+\cdots+\ell_m=k-m}}
X^{\ell_0}
E_{q_{i_1+1},q_{i_1}}
X^{\ell_1}
\cdots
E_{q_{i_m+1},q_{i_m}}
X^{\ell_m},
\]
where \(q_{r+1}=q_0\).
\end{enumerate}
\end{lemma}
\begin{proof}
Expand
\[
M_q(u)^k
=
\left(
	X+\sum_{j=1}^{r}u_jE_{q_{j+1},q_j}
\right)^k
\]
as a noncommutative product of $k$ factors.

To obtain the monomial
\[
u^I=u_1^{\rho_1}\cdots u_r^{\rho_r},
\]
one must choose exactly
\[
m=\rho_1+\cdots+\rho_r
\]
perturbation factors. Reading these perturbation factors from left to
right gives a word
\[
(i_1,\dots,i_m)\in\Ord(I).
\]

The remaining $k-m$ factors are copies of $X$. They are distributed
among the $m+1$ gaps before, between, and after the perturbation
factors. Hence they determine nonnegative integers
\[
\ell_0,\dots,\ell_m
\]
satisfying
\[
\ell_0+\cdots+\ell_m=k-m.
\]

The corresponding matrix word is
\[
	X^{\ell_0}
	E_{q_{i_1+1},q_{i_1}}
	X^{\ell_1}
	\cdots
	E_{q_{i_m+1},q_{i_m}}
	X^{\ell_m}.
\]
Summing over all words in $\Ord(I)$ and all weak compositions
$\ell_0+\cdots+\ell_m=k-m$ proves the formula.

If $m>k$, one cannot choose $m$ perturbation factors from a product of
length $k$, so the coefficient is zero. If $m=0$, no perturbation factor
is chosen, and the coefficient is $X^k$.
\end{proof}

When \(X\) is diagonal, the preceding coefficient formula can be written in closed form.  We use the following notation.
For variables \(y_0,\dots,y_m\) and \(d\in\bbZ_{\geq 0}\), define the {\it complete
homogeneous symmetric polynomial}
\[
h_d(y_0,\dots,y_m)
=
\sum_{\substack{\ell_0,\dots,\ell_m\geq 0\\
\ell_0+\cdots+\ell_m=d}}
y_0^{\ell_0}\cdots y_m^{\ell_m}.
\]
Also, set
\[
	h_0(y_0,\dots,y_m)=1,
	\qquad
	h_d(y_0,\dots,y_m)=0 \quad \text{for } d<0.
\]
The next formula follows by specializing the preceding lemma to a diagonal matrix.

\begin{corollary}
Let
\[
X=\operatorname{diag}(x_1,\dots,x_n),
\qquad
I=(\rho_1,\dots,\rho_r)\in\bbZ_{\geq 0}^r,
\qquad
m=|I|,
\qquad k\in\bbZ_{\geq0}.
\]
\begin{enumerate}
\item[$(1)$]
If \(I=(0,\dots,0)\), then
\[
\Gamma_{q,0}^{(k)}(X)=X^k.
\]

\item[$(2)$]
If \(1\leq m\), then
\[
\begin{aligned}
\Gamma_{q,I}^{(k)}(X)
&=
\sum_{i=(i_1,\ldots,i_m)\in\Ord (I)}
\left(
\prod_{a=1}^{m-1}
\delta_{q_{i_a},q_{i_{a+1}+1}}
\right)
\\
&\qquad\times
h_{k-m}\left(
x_{q_{i_1+1}},
x_{q_{i_1}},
x_{q_{i_2}},
\ldots,
x_{q_{i_m}}
\right)
E_{q_{i_1+1},q_{i_m}}.
\end{aligned}
\]
\end{enumerate}
\end{corollary}

For \(m=1\), the empty product
\(\prod_{a=1}^{0}(\cdots)\)
is understood to be \(1\).  Setting \(i_{m+1}=i_1\), we obtain the following trace formula.

\begin{corollary}\label{coro: trace gamma}
Let
\[
X=\operatorname{diag}(x_1,\ldots,x_n),
\qquad
I=(\rho_1,\ldots,\rho_r)\in\bbZ_{\geq0}^r,
\qquad
m=|I|,
\qquad k\in\bbZ_{\geq0}.
\]
\begin{enumerate}
\item[$(1)$]
If \(I=(0,\ldots,0)\), then
\[
\Tr\left(\Gamma_{q,0}^{(k)}(X)\right)
=
\Tr(X^k)
=
\sum_{a=1}^n x_a^k.
\]

\item[$(2)$]
If \(1\leq m\), then, 
\[
\begin{aligned}
\Tr\left(\Gamma_{q,I}^{(k)}(X)\right)
&=
\sum_{\bfi=(i_1,\ldots,i_m)\in\Ord(I)}
\left(
\prod_{a=1}^{m}
\delta_{q_{i_a},q_{i_{a+1}+1}}
\right)\\
&\qquad\times
h_{k-m}\left(
 x_{q_{i_1+1}},
 x_{q_{i_1}},
 x_{q_{i_2}},
 \ldots,
 x_{q_{i_m}}
\right).
\end{aligned}
\]
\end{enumerate}
\end{corollary}
The  index \(i_{a+1}+1\) may equal \(r+1\),
which is incorrect.
In the graph notation below, we represent this successor by the
vertex \(0\), in accordance with \(q_{r+1}=q_0\) in
\eqref{for: trace w_sr}.  Thus, for \(1\leq j<r\), the successor
of \(j\) is \(j+1\), and the successor of \(r\) is \(0\).

We record two elementary examples that illustrate the preceding coefficient formula.

\begin{example}
Let \(X=\operatorname{diag}(x_1,\dots,x_n)\).  Write
\(\varepsilon_j\in\bbZ_{\geq 0}^r\) for the standard multi-index whose
\(j\)-th entry is \(1\) and whose other entries are \(0\).

First, let \(k\geq1\) and take \(I=\varepsilon_j\).  Then \(m=1\), and
\(\Ord(I)=\{(j)\}\).  Corollary~\ref{coro: trace gamma} gives
\[
\Tr \left(
\Gamma_{q,\varepsilon_j}^{(k)}(X)
\right)
=
\delta_{q_j,q_{j+1}}
h_{k-1}(x_{q_{j+1}},x_{q_j})
=
k\delta_{q_j,q_{j+1}}x_{q_j}^{k-1}.
\]

This is the coefficient form of the elementary identity
\[
   \left.
\frac{\partial}{\partial u_j}
\Tr\!\left(
\bigl(X+u_jE_{q_{j+1},q_j}\bigr)^k
\right)
\right|_{u_j=0}
   =
   k\,\delta_{q_{j},q_{j+1}}\,x_{q_j}^{k-1}.
\]
\end{example}

\begin{example}
Assume \(r=2\) and take
\(I=(1,1)=\varepsilon_1+\varepsilon_2\).  Then
\[
   \Ord(I)=\{(1,2),(2,1)\}.
\]
For \(k\geq 2\), Corollary~\ref{coro: trace gamma} gives
\[
\begin{aligned}
\Tr \left(
\Gamma_{q,(1,1)}^{(k)}(X)
\right)
&=
\delta_{q_0,q_1}
h_{k-2}(x_{q_2},x_{q_1},x_{q_2})
\\
&\quad+
\delta_{q_0,q_1}
h_{k-2}(x_{q_0},x_{q_2},x_{q_1}).
\end{aligned}
\]
In particular, for \(k=2\),
\[
	\Tr \left(
		\Gamma_{q,(1,1)}^{(2)}(X)
		\right)
		=
		2\delta_{q_0,q_1}.
\]
\end{example}

The preceding identities provide the matrix coefficients required in the proof of Theorem \ref{thm: radial part}; in particular, Corollary \ref{coro: trace gamma} gives the diagonal trace coefficients. It remains to reorganize the resulting sums and prove that the coefficients
\[
\sum_q {\delta_{q_0, q_1}}x_{q_0}^s
\Tr(\Gamma_{q,\pi}^{\alpha}(x_1,\dots,x_n))
\]
are the power-sum polynomials asserted in Theorem \ref{thm: radial part}.  We first examine the case \(r=1\).

\begin{example}
For \(r=1\), the only set partition of \([1]\) is
\(\pi=\{\{1\}\}\).
Set \(\alpha(\{1\})=a\).
Corollary~\ref{coro: trace gamma}
gives
\[
   \Tr\left(\Gamma_{q,I_{\{1\}}}^{(a)}(X)\right)
   =
   \delta_{q_0,q_1}\,a\,x_{q_1}^{a-1}.
\]
Hence the formula gives
\[
\begin{aligned}
   \hc'\bigl(\Tr^q(X^sD)\bigr)
   &=
   \sum_{a=1}^n
   \left(
      \sum_{q_0,q_1=1}^n
      \delta_{q_0,q_1}x_{q_0}^s
      \delta_{q_0,q_1}\,a\,x_{q_1}^{a-1}
   \right)
   \partial_{p_a}  \\
   &=
   \sum_{a=1}^n
   a\,p_{a+s-1}\,\partial_{p_a}.
\end{aligned}
\]

On the other hand, the chain rule gives
\[
   \hc'\bigl(\Tr^q(X^sD)\bigr)
   =
   {\sum_{i=1}^n x_i^s\partial_{x_i}}
   =
   \sum_{a=1}^n
   a\,p_{a+s-1}\,\partial_{p_a}.
\]
\end{example}

We now introduce the combinatorial notation used to treat the general
coefficient.
For a set partition \(\pi\) of \([r]\) and a map
\(\alpha:\pi\to\{1,\dots,n\}\), define
\[
\Ord(\pi):=\prod_{B\in\pi}\Ord(I_B);
\]
for every block \(B\in\pi\), write
\(\bfi^B=(i_1^B,\ldots,i_{|B|}^B),\ i_{|B|+1}^B:=i_1^B\);
and
\[
\Lambda(\pi,\alpha):=
\left\{\bfl=(\ell_t^B)_{B\in\pi,\,0\leq t\leq |B|}\ \middle|\
\begin{array}{l}
 \ell_t^B\in\bbZ_{\geq0},\\
 \sum_{t=0}^{|B|}\ell_t^B=\alpha(B)-|B|\text{ for every }B\in\pi
\end{array}
\right\}.
\]
If \(\alpha(B)<|B|=|I_B|\) for some block \(B\), this set is empty, in agreement
with the convention \(h_d=0\) for \(d<0\).

Using Corollary \ref{coro: trace gamma} and expanding the complete homogeneous
symmetric polynomials, we obtain
\begin{equation}\label{for: coeff expand}
	\begin{aligned}
		&\sum_q
		\delta_{q_0,q_1}x_{q_0}^{s}
		\Tr \left(
		\Gamma_{q,\pi}^{\alpha}(x_1,\ldots,x_n)
		\right)
		\\
		&=
		\sum_{\bfi\in\Ord (\pi)}
		\sum_{\bfl\in\Lambda(\pi,\alpha)}
		\sum_q
		\delta_{q_0,q_1}x_{q_0}^{s}
		\prod_{B\in\pi}
		\left[
		\left(
		\prod_{t=1}^{|B|}
		\delta_{q_{i_t^B},q_{i_{t+1}^B+1}}
		\right)
		x_{q_{i_1^B+1}}^{\ell_0^B}
		\prod_{t=1}^{|B|}
		x_{q_{i_t^B}}^{\ell_t^B}
		\right].
		\end{aligned}
\end{equation}

Fix \(\pi\), \(\alpha\), \(\bfi\in\Ord(\pi)\), and
\(\bfl\in\Lambda(\pi,\alpha)\).  Define \(G_{\pi,\bfi}\) to be the graph
with vertices \(\{0,1,\ldots,r\}\), allowing loops and
repeated edges, whose edges are the distinguished edge \(\{0,1\}\) and
\[
\{i_t^B,i_{t+1}^B+1\},
\qquad B\in\pi,
\quad 1\leq t\leq|B|.
\]
All successor labels are understood cyclically according to the
convention above.  Let \(\mC(G_{\pi,\bfi})\) denote its set of connected
components.

Appendix \ref{sect:examples} gives an \(r=4\) example illustrating how the graph is assembled from the combinatorial data.

{For \(j\in\{0,1,\dots,r\}\), define
\[
\lambda_j
=
s\ind_{\{0\}}(j)
+
\sum_{B\in\pi}
\left(
\ell_0^B\ind_{\{i_1^B +1\}}(j)
+
\sum_{t=1}^{|B|}
\ell_t^B\ind_{\{i_t^B\}}(j)
\right),
\]
where \(\ind_A(j)=1\) if \(j\in A\) and \(\ind_A(j)=0\) otherwise.  For each connected component \(C\in\mC(G_{\pi,\bfi})\), set
\[
L_C:=\sum_{j\in C}\lambda_j.
\]}

The following identity explains precisely how the Kronecker-delta
constraints produce power sums.

\begin{proposition}[Power-sum coefficient formula]\label{prop: pow sum coeff}
	With the notation above,
	\begin{equation}
		\sum_{q_0,\ldots,q_r=1}^{n}
			\delta_{q_0,q_1}x_{q_0}^{s}
			\prod_{B\in\pi}
			\left[
			\left(
			\prod_{t=1}^{|B|}
			\delta_{q_{i_t^B},q_{i_{t+1}^B + 1}}
			\right)
			x_{q_{i_1^B + 1}}^{\ell_0^B}
			\prod_{t=1}^{|B|}
			x_{q_{i_t^B}}^{\ell_t^B}
			\right]
			=
			\prod_{C\in\mC(G_{\pi,\bfi})}
			p_{L_C}(x_1,\ldots,x_n).
	\end{equation}
\end{proposition}

\begin{proof}
The product of Kronecker deltas on the left-hand side imposes exactly the equalities \(q_a=q_b\) along the edges \(\{a,b\}\) of \(G_{\pi,\bfi}\).
Thus this Kronecker delta condition implies choosing one index \(a_C\in\{1,\dots,n\}\) for each connected component \(C\in\mC(G_{\pi,\bfi})\) and setting \(q_j=a_C\) for every \(j\in C\).
By the definitions of \(\lambda_j\) and \(L_C\),
\[
	x_{q_0}^{s}
	\prod_{B\in\pi}
	\left[
	x_{q_{i_1^B + 1}}^{\ell_0^B}
	\prod_{t=1}^{|B|}
	x_{q_{i_t^B}}^{\ell_t^B}
	\right]
=
\prod_{j=0}^r x_{q_j}^{\lambda_j}
=
	\prod_{C\in\mC(G_{\pi,\bfi})}
	\prod_{j\in C}x_{a_C}^{\lambda_j}
	=
	\prod_{C\in\mC(G_{\pi,\bfi})}
	x_{a_C}^{L_C}.
\]
Therefore
\[
	\begin{aligned}
		&\sum_{q_0,\ldots,q_r=1}^{n}
		\delta_{q_0,q_1}x_{q_0}^{s}
		\prod_{B\in\pi}
		\left[
		\left(
		\prod_{t=1}^{|B|}
		\delta_{q_{i_t^B},q_{i_{t+1}^B + 1}}
		\right)
		x_{q_{i_1^B + 1}}^{\ell_0^B}
		\prod_{t=1}^{|B|}
		x_{q_{i_t^B}}^{\ell_t^B}
		\right] \\
		&=
		\sum_{(a_C)_{C\in\mC(G_{\pi,\bfi})}}
		\prod_{C\in\mC(G_{\pi,\bfi})}
		x_{a_C}^{L_C} \\
		&=
		\prod_{C\in\mC(G_{\pi,\bfi})}
		\left(\sum_{a_C=1}^{n}x_{a_C}^{L_C}\right) \\
		&=
		\prod_{C\in\mC(G_{\pi,\bfi})}
		p_{L_C}(x_1,\dots,x_n).
	\end{aligned}
\]
This proves the formula.
\end{proof}

Applying Proposition \ref{prop: pow sum coeff} to \eqref{for: coeff expand}, we obtain
\begin{equation}
	\sum_{q} {\delta_{q_0, q_1}}x_{q_0}^s
	\Tr(\Gamma_{ q, \pi} ^\alpha (x_1, \dots, x_n))
	=
	\sum_{\bfi \in \Ord(\pi)}
	\sum_{\bfl \in \Lambda(\pi, \alpha)}
	\Bigl(
	\prod_{C\in\mC(G_{\pi,\bfi})}
	p_{L_C}(x_1,\dots,x_n)
	\Bigr).
\end{equation}
This identity gives a finite algorithm for every coefficient.  Example \ref{Appexample:r=4contd} illustrates the computation when \(r=4\).

We now complete the proof of Theorem \ref{thm: radial part}.

\begin{proof}[Proof of Theorem \ref{thm: radial part}]
Let \(f\in\bbC[\gl_n (\bbC)]^{\gl_n (\bbC)}\) and write
\[
	f|_{\fkh}=F(p_1,\dots,p_n).
\]
The polynomial
\[
X\longmapsto F\bigl(\Tr(X),\Tr(X^2),\ldots,\Tr(X^n)\bigr)
\]
is \(\gl_n (\bbC)\)-invariant and has the same restriction to \(\fkh\) as
\(f\).  By the injectivity of Chevalley restriction,
\[
f(X)=F\bigl(\Tr(X),\Tr(X^2),\ldots,\Tr(X^n)\bigr)
\qquad (X\in\gl_n (\bbC)).
\]
Consequently, by the definition of the quantum trace operator and the
choice \(D=(\partial_{x_{ji}})\), the action of
\(\Tr^q(X^sD^r)\) on \(f\), restricted to \(\fkh\), is
\[
	\bigl(\Tr^q(X^sD^r)f\bigr)|_{\fkh}
	=
	\sum_q \hbar^r
	{\delta_{q_0, q_1}} x_{q_0}^s
	\left.
	\frac{\partial ^r F\bigl(\Tr(M_q(u)),\dots,\Tr(M_q(u)^n)\bigr)}{\partial u_1 \dots \partial u_r}
	\right|_{u=0}.
\]

Here the factor \(\delta_{q_0, q_1}\) comes from restricting the matrix entry \((X^s)_{q_0, q_1}\) to the diagonal Cartan subalgebra.  
Set \(g_a(u)=\Tr(M_q(u)^a)\) for \(1\leq a\leq n\).  Applying the multivariate Fa\`a di Bruno formula 
\eqref{Brunosformula}
in Appendix \ref{app: fdb}, we obtain
\[
	\left.
	\frac{\partial ^r F\bigl(g_1(u),\dots,g_n(u)\bigr)}{\partial u_1 \dots \partial u_r}
	\right|_{u=0}
	=
	\sum_{(\pi,\alpha)}
	(\pa_v^\alpha F)(g(0))
	\prod_{B\in\pi}
	\left.
	\left(\prod_{j\in B}\frac{\partial}{\partial u_j}\right)
	\Tr\bigl(M_q(u)^{\alpha(B)}\bigr)
	\right|_{u=0}.
\]
By the squarefree coefficient identity above,
\[
\left.
\left(
\prod_{j\in B}\frac{\partial}{\partial u_j}
\right)
\Tr \left(
M_q(u)^{\alpha(B)}
\right)
\right|_{u=0}
=
\Tr \left(
\Gamma_{q,I_B}^{(\alpha(B))}
\right).
\]

Therefore
\begin{equation}\label{for: radial part coeff}
\begin{aligned}
\bigl(\Tr^q(X^sD^r)f\bigr)|_{\fkh}
&=
\sum_{(\pi,\alpha)}\hbar^{r-|\pi|}
\left(
\sum_q\delta_{q_0,q_1}x_{q_0}^s
\Tr\bigl(\Gamma_{q,\pi}^{\alpha}(x_1,\ldots,x_n)\bigr)
\right)
\partial_p^\alpha(F)\\
&=
\sum_{(\pi,\alpha)}\hbar^{r-|\pi|}
\left\{
\sum_{\bfi\in\Ord(\pi)}
\sum_{\bfl\in\Lambda(\pi,\alpha)}
\prod_{C\in\mC(G_{\pi,\bfi})}
p_{L_C}(x_1,\ldots,x_n)
\right\}
\partial_p^\alpha(F).
\end{aligned}
\end{equation}
Note that
\[
\hbar^r
\left.
\left(
\prod_{B\in\pi}
\frac{\partial}{\partial v_{\alpha(B)}}
\right)
F(v_1,\ldots,v_n)
\right|_{(v_1,\ldots,v_n)=(p_1,\ldots,p_n)}
=
\hbar^{r-|\pi|}
\partial_p^\alpha(F),
\]
where \(\partial_p^\alpha\) contains \(|\pi|\) \(\hbar\)-scaled
derivations.  By the defining property of \(\hc'\), this is the asserted formula for \(\hc'(\Tr^q(X^sD^r))\).
This completes the proof.
\end{proof}

\subsection{Harish--Chandra homomorphism in power-sum coordinates}
\label{subsec: hc power sum formula}

We now pass from the radial-part formula of Theorem~\ref{thm: radial part} to the Harish--Chandra homomorphism.
For Harish--Chandra conjugation, \(\partial_{p_a}\) also denotes its
unique \(S_n\)-invariant lift to \(\fkh_{\reg}\); this lift exists
because \(\fkh_{\reg}\to\fkh_{\reg}/S_n\) is finite étale.  Define
\[
\delta\,\partial_{p_a}\,\delta^{-1}
=
\partial_{p_a}-\frac{1}{2}\partial_{p_a}(\ln\delta^2),
\qquad 1\leq a\leq n.
\]
The last expression is algebraic on \(\fkh_{\reg}\).  Moreover,
\[
\delta[\partial_{p_a},\partial_{p_b}]\delta^{-1}
=0,
\]
so the product
\(\prod_{B\in\pi}\delta\,\partial_{p_{\alpha(B)}}\,\delta^{-1}\)
is independent of the ordering of the blocks \(B\).

Next, we write the coefficients in \eqref{for: radial part coeff} as follows.
\begin{equation}\label{eq:C-pi-alpha-def-hc}
   C_{\pi,\alpha}^{s,r}(p)
   :=
   \sum_{\bfi\in\Ord(\pi)}
   \sum_{\bfl\in\Lambda(\pi,\alpha)}
   \prod_{C\in\mC(G_{\pi,\bfi})}p_{L_C}.
\end{equation}

\begin{remark}\label{rmk: algorithm for coeff}
We summarize the steps of our algorithm for computing the coefficients as follows.
\begin{enumerate}
\item Choose \(\bfi\in\Ord(\pi)\).
\item Construct \(G_{\pi,\bfi}\), including the distinguished edge
      \(\{0,1\}\).
\item Determine the connected components
      \(C\in\mC(G_{\pi,\bfi})\).
\item For each \({\bf l} \in \Lambda (\pi, \alpha)\),
	assign the vertex weights \(\lambda_j\).
\item For every component, compute
      \(L_C=\sum_{j\in C}\lambda_j\).
\item Record the graph contribution
      \(\prod_{C\in\mC(G_{\pi,\bfi})}p_{L_C}\).
\item Sum these contributions over
      \(\bfi\in\Ord(\pi)\) and
      \(\bfl\in\Lambda(\pi,\alpha)\).
\end{enumerate}
\end{remark}

Theorem \ref{thm: radial part}, together with the construction in Section \ref{subsec: tilde Tr}, yields the following formula for the Harish--Chandra homomorphism.

\begin{proposition}[Harish--Chandra homomorphism in power sums]
	\label{prop: hc formula power sums}
	Let \(\fkg=\gl_n(\bbC)\) and let \(\fkh\) be the Cartan subalgebra
	of diagonal matrices. For \(s\in\bbZ_{\geq0}\) and
	\(r\in\bbZ_{\geq1}\),
	\[
	\hc(\Tr^q(X^sD^r))
	=
	\sum_{(\pi,\alpha)}\hbar^{r-|\pi|}
	C_{\pi,\alpha}^{s,r}(p)
	\prod_{B\in\pi}
	\left(
	\partial_{p_{\alpha(B)}}
	-
	\frac12\partial_{p_{\alpha(B)}}(\ln\delta^2)
	\right).
	\]
\end{proposition}
\begin{proof}
In \(\mathcal D_\hbar(\fkh_{\reg})^{S_n}\),
Theorem~\ref{thm: radial part} gives
\[
\hc'\bigl(\Tr^q(X^sD^r)\bigr)
=
\sum_{(\pi,\alpha)}\hbar^{r-|\pi|}
C_{\pi,\alpha}^{s,r}(p)
\prod_{B\in\pi}\partial_{p_{\alpha(B)}}.
\]
Multiplication by \(C_{\pi,\alpha}^{s,r}(p)\) commutes with
multiplication by \(\delta\), and conjugation is an algebra
automorphism.  Hence
\[
\begin{aligned}
\hc\bigl(\Tr^q(X^sD^r)\bigr)
&=\delta\,\hc'\bigl(\Tr^q(X^sD^r)\bigr)\,\delta^{-1}\\
&=
\sum_{(\pi,\alpha)}\hbar^{r-|\pi|}
C_{\pi,\alpha}^{s,r}(p)
\prod_{B\in\pi}
\bigl(\delta\partial_{p_{\alpha(B)}}\delta^{-1}\bigr)\\
&=
\sum_{(\pi,\alpha)}\hbar^{r-|\pi|}
C_{\pi,\alpha}^{s,r}(p)
\prod_{B\in\pi}\left(
	\partial_{p_{\alpha(B)}}
	-
	\frac12\partial_{p_{\alpha(B)}}(\ln\delta^2)
	\right).
\end{aligned}
\]
The factors commute by the calculation above.  Although individual
summands may have poles, Theorem~\ref{thm: hc regular} implies that the
total operator extends across \(\delta=0\) and belongs to
\(\mathcal D_\hbar(\fkh)^{S_n}\).
\end{proof}

We next examine several low-order cases, beginning with first-order quantum trace operators.

\begin{proposition}
	Let \(\fkg = \gl_n (\bbC)\) and let \(\fkh\) be the Cartan subalgebra of diagonal matrices. 
	Let
	\[
		R_0(p)=0,\quad
		\text{and}\quad
		R_s(p)
		=
		\frac12
		\sum_{u=0}^{s-1}
		\bigl(p_u p_{s-1-u}-p_{s-1}\bigr) \text{ for \(s\geq1\).}
	\]
	{Then, for every \(s\in\bbZ_{\geq0}\),}
	\[
   \hc\bigl(\Tr^q(X^sD)\bigr)
   =
   \sum_{a=1}^n
   a\,p_{s+a-1}\partial_{p_a}
   -
   \hbar R_s(p).
   \]
\end{proposition}

\begin{proof}
We first compute the complete coefficient in
\eqref{eq:C-pi-alpha-def-hc} from the graph, for every
\(s\geq0\), and then evaluate the Harish--Chandra conjugation term.

For \(r=1\), the only partition is \(\pi=\{B\}\), where
\(B=\{1\}\). Fix \(\alpha(B)=a\), with \(1\leq a\leq n\).
There is only one ordering, \(\bfi^B=(1)\), and the admissible
weights satisfy
\[
\ell_0^B+\ell_1^B=a-1,
\qquad \ell_0^B,\ell_1^B\geq0.
\]
The cyclic successor of \(1\) is \(0\). Thus the distinguished
edge \(\{0,1\}\) and the block edge \(\{1,0\}\) are parallel,
and the graph has the single connected component \(C=\{0,1\}\).
The vertex and component weights are
\[
(\lambda_0,\lambda_1)=(s+\ell_0^B,\ell_1^B),
\qquad L_C=s+\ell_0^B+\ell_1^B=s+a-1.
\]
Consequently every admissible pair contributes \(p_{s+a-1}\).
The pairs are \((\ell_0^B,\ell_1^B)=(u,a-1-u)\),
\(0\leq u\leq a-1\), so their number is exactly \(a\), and
\[
C_{\pi,\alpha}^{s,1}(p)=a\,p_{s+a-1}.
\]

Since \(\hbar^{r-|\pi|}=1\), Proposition
\ref{prop: hc formula power sums} now gives
\[
\hc\bigl(\Tr^q(X^sD)\bigr)
=\sum_{a=1}^n a\,p_{s+a-1}\partial_{p_a}
-\frac12\sum_{a=1}^n a\,p_{s+a-1}
\partial_{p_a}(\ln\delta^2).
\]
Only the last, scalar term remains to be evaluated. On
\(\fkh_{\reg}\), the chain rule gives
\[
\sum_{a=1}^n a\,p_{s+a-1}\partial_{p_a}
=\sum_{i=1}^n x_i^s\partial_{x_i}.
\]
A straightforward calculation shows that
\[
\frac12\sum_{a=1}^n a\,p_{s+a-1}
\partial_{p_a}(\ln\delta^2)
=\hbar\sum_{1\leq i<j\leq n}
\frac{x_i^s-x_j^s}{x_i-x_j}.
\]
For \(s=0\), every numerator is zero, so this equals
\(\hbar R_0(p)=0\). For \(s\geq1\), the finite geometric-series
identity yields
\[
\begin{aligned}
2\sum_{1\leq i<j\leq n}\frac{x_i^s-x_j^s}{x_i-x_j}
&=\sum_{i\neq j}\sum_{u=0}^{s-1}x_i^{s-1-u}x_j^u\\
&=\sum_{u=0}^{s-1}
\bigl(p_{s-1-u}p_u-p_{s-1}\bigr)
=2R_s(p).
\end{aligned}
\]
Thus, for every \(s\geq0\),
\[
\frac12\sum_{a=1}^n a\,p_{s+a-1}
\partial_{p_a}(\ln\delta^2)=\hbar R_s(p).
\]
Substitution proves the asserted formula.
\end{proof}

\begin{example}[The Euler operator]
Take \(s=1\) in the first-order formula.  Since
\[
	R_1(p)=\frac{n(n-1)}2,
\]
we obtain
\[
	\hc\bigl(\Tr^q(XD)\bigr)
	=
	\sum_{a=1}^n a\,p_a\partial_{p_a}
	-
	\hbar\frac{n(n-1)}2.
\]

\end{example}

\begin{example}
Let \(n=s=2\).  The formula gives
\[
   \hc\bigl(\Tr^q(X^2D)\bigr)
   =
   p_2\partial_{p_1}
   +
   2p_3\partial_{p_2}
   -
   \hbar p_1.
\]

\end{example}

We next turn to second-order quantum trace operators, focusing on the Laplacian.

\begin{proposition}
Let \(\fkg = \gl_n (\bbC)\) and let \(\fkh\) be the Cartan subalgebra of diagonal matrices. 
Then the image of \(\Tr^q(D^2)\) under the Harish--Chandra homomorphism is
\begin{equation}
		\hc(\Tr^q(D^2))
	   =\sum_{a,b=1}^n
	   ab\,p_{a+b-2}
	   \partial_{p_a}\partial_{p_b}
	   +
	   \hbar
	   \sum_{a=2}^n
	   a(a-1)p_{a-2}\partial_{p_a}.
\end{equation}
\end{proposition}

\begin{proof}
We compute all coefficients from the graph side of Proposition
\ref{prop: pow sum coeff}, using \eqref{eq:C-pi-alpha-def-hc} with
\(s=0\) and \(r=2\), and then perform \(\delta\)-conjugation.
There are exactly two set partitions of \([2]\).

First take \(\pi_{\mathrm{disc}}=\{\{1\},\{2\}\}\), with
\(\alpha(\{1\})=a\) and \(\alpha(\{2\})=b\).
The singleton orderings are fixed, and
\[
\ell_0^{\{1\}}+\ell_1^{\{1\}}=a-1,
\qquad
\ell_0^{\{2\}}+\ell_1^{\{2\}}=b-1.
\]
The graph edges are \(\{0,1\}\), \(\{1,2\}\), and \(\{2,0\}\),
where the last edge uses the cyclic successor of \(2\), namely \(0\).
Hence \(C=\{0,1,2\}\) is the only connected component. Its weights are
\[
\begin{aligned}
(\lambda_0,\lambda_1,\lambda_2)
&=\bigl(\ell_0^{\{2\}},\ell_1^{\{1\}},
\ell_0^{\{1\}}+\ell_1^{\{2\}}\bigr),\\
L_C&=(a-1)+(b-1)=a+b-2.
\end{aligned}
\]
Therefore,
it contributes
\(p_{a+b-2}\), and
\[
C_{\pi_{\mathrm{disc}},\alpha}^{0,2}(p)=ab\,p_{a+b-2}.
\]

Next take \(\pi_{\mathrm{blk}}=\{B\}\), with \(B=\{1,2\}\)
and \(\alpha(B)=a\). For \(a=1\), the set
\(\Lambda(\pi_{\mathrm{blk}},\alpha)\) is empty, so the coefficient
is zero. For \(a\geq2\), the admissible triples satisfy
\[
\ell_0^B+\ell_1^B+\ell_2^B=a-2,
\qquad \ell_0^B,\ell_1^B,\ell_2^B\geq0,
\]
and both orderings \((1,2)\) and \((2,1)\) must be counted.
Together with the distinguished edge \(\{0,1\}\), they give the
same multigraph, with connected components
\[
C_1=\{0,1\},\qquad C_2=\{2\}.
\]
Now, we compute weights:
\[
\begin{array}{c|c|c}
\bfi^B & (\lambda_0,\lambda_1,\lambda_2)&(L_{C_1},L_{C_2})\\ \hline
(1,2)&(0,\ell_1^B,\ell_0^B+\ell_2^B)
     &(\ell_1^B,\ell_0^B+\ell_2^B)\\
(2,1)&(\ell_0^B,\ell_2^B,\ell_1^B)
     &(\ell_0^B+\ell_2^B,\ell_1^B)
\end{array}
\]
Fix \(u=L_{C_1}\), with \(0\leq u\leq a-2\).
For \((1,2)\), one has \(\ell_1^B=u\), leaving
\(a-1-u\) choices for the pair with
\(\ell_0^B+\ell_2^B=a-2-u\).
For \((2,1)\), one has \(\ell_1^B=a-2-u\), leaving
\(u+1\) choices for the pair with \(\ell_0^B+\ell_2^B=u\).
Every one of these choices contributes \(p_u p_{a-2-u}\).
Consequently the coefficient is
\[
\begin{aligned}
C_{\pi_{\mathrm{blk}},\alpha}^{0,2}(p)
&=\sum_{u=0}^{a-2}\bigl((a-1-u)+(u+1)\bigr)p_u p_{a-2-u}\\
&=a\sum_{u=0}^{a-2}p_u p_{a-2-u}.
\end{aligned}
\]
Here \(p_0=n\); in particular, \(a=2\) gives \(2p_0^2=2n^2\).

The powers \(\hbar^{2-|\pi|}\) are respectively \(1\) and
\(\hbar\). Thus Theorem~\ref{thm: radial part}, in the form
\eqref{for: radial part coeff}, gives
\begin{equation}\label{eq:hc-D2-V-form}
\hc'\bigl(\Tr^q(D^2)\bigr)
=\sum_{a,b=1}^n ab\,p_{a+b-2}\partial_{p_a}\partial_{p_b}
+\hbar\sum_{a=2}^n a
\left(\sum_{u=0}^{a-2}p_u p_{a-2-u}\right)\partial_{p_a}.
\end{equation}

Conjugate \eqref{eq:hc-D2-V-form} by \(\delta\), replacing each
\(\partial_{p_a}\) by
\(\partial_{p_a}-\frac12\partial_{p_a}(\ln\delta^2)\).
A straightforward calculation shows that
\[
\hc\bigl(\Tr^q(D^2)\bigr)
=\sum_{a,b=1}^n ab\,p_{a+b-2}\partial_{p_a}\partial_{p_b}
+\hbar\sum_{a=2}^n a(a-1)p_{a-2}\partial_{p_a},
\]
as asserted, consistently with Proposition~\ref{prop: hc laplacian}.
\end{proof}

\begin{example}[The Laplacian for \(n=2\)]
Let \(n=2\).
The preceding proposition gives
\[
\hc\bigl(\Tr^q(D^2)\bigr)
=
2\partial_{p_1}^2
+4p_1\partial_{p_1}\partial_{p_2}
+4p_2\partial_{p_2}^2
+4\hbar\partial_{p_2}.
\]
\end{example}

Note that power-sum realizations of differential operators are
well established,
see \cite{SerVes2015,Dra2024Gen} for examples.
The new point of Theorem \ref{thm: radial part} is a closed-form,
finite-\(n\), and nonrecursive power-sum expansion for arbitrary \(s\) and \(r\), in which coefficients \(C_{\pi,\alpha}^{s,r}\) are
characterized by combinatorial data such as set partitions and graphs.

More broadly,
these formulas provide a concrete starting point and may reveal a structural connection between the noncommutative geometry on quiver algebras
and the deformed \(W_{1+\infty}\), affine-Yangian, and
Calogero--Moser systems \cite{ArbSch2013,KN2016Qua,SchVas2013,DesHal2012}.

\section{From quantum preprojective to rational Cherednik}\label{sec: comparison}

In this section, we construct a homomorphism from the quantum preprojective algebra of the Jordan quiver to a localized spherical rational Cherednik algebra.
 We then express the resulting map explicitly in power-sum coordinates. Our principal reference for rational Cherednik algebras is Etingof--Ginzburg \cite{EG2002}.

\subsection{Rational Cherednik algebra}

Let \(\fkg\) be a finite-dimensional complex reductive Lie algebra with Cartan subalgebra \(\fkh\), root system \(R\subset\fkh^*\), and Weyl group \(W\). Fix a set of positive roots \(R^+\subset R\), and set
\[
\fkh_{\reg}=\fkh\setminus\bigcup_{\alpha\in R}\ker(\alpha),
\qquad
\delta=\prod_{\alpha\in R^+}\alpha.
\]
Let \(c:R\to\bbC\) be \(W\)-invariant, so that \(c_\alpha=c_{-\alpha}\).

\begin{definition}[Etingof--Ginzburg]\label{def: rca}
The rational Cherednik algebra \(\rmH_{\hbar,c}^{\mathrm{pol}}(\fkh,W)\) associated to the data
\(\fkh, W\) is defined to be the
\(\bbC[\hbar]\)-algebra generated by
\(\fkh\oplus\fkh^*\) and \(W\), subject to the following relations:
\[\begin{array}{ll}
w x w^{-1}=w(x),\qquad w y w^{-1}=w(y),
&w\in W,\ y\in\fkh,\ x\in\fkh^*,\\[1mm]
[x_1,x_2]=[y_1,y_2]=0,
&y_1,y_2\in\fkh,\ x_1,x_2\in\fkh^*,\\[1mm]
\displaystyle
[y,x]=\hbar\langle y,x\rangle-\frac{1}{2}\sum_{\alpha\in R}c_\alpha
\langle y,\alpha\rangle\langle\alpha^\vee,x\rangle s_\alpha,
&y\in\fkh,\ x\in\fkh^*.
	\end{array}
    \]
\end{definition}

For \(\fkg=\gl_n(\bbC)\), we write
\(\rmH_{\hbar,c}^\pol (n)=\rmH_{\hbar,c} ^\pol (\bbC^n,S_n)\). It is generated by
\(\bbC S_n\), commuting elements \(x_1,\dots,x_n\), and commuting elements
\(y_1,\dots,y_n\), with
\[
s_{ij}x_k=x_{s_{ij}(k)}s_{ij},
\qquad
s_{ij}y_k=y_{s_{ij}(k)}s_{ij},
\]
and
\[
\begin{array}{ll}
\relax [x_i,x_j]=[y_i,y_j]=0, &\text{for all }i,j;\\[1mm]
\relax [y_i,x_j]=c\,s_{ij}, &\text{for }i\neq j;\\[1mm]
\relax 
\displaystyle [y_i,x_i]=\hbar-c\sum_{k\neq i}s_{ik}, &\text{for all }i.
\end{array}
\]
Let
\[
\be=|W|^{-1}\sum_{w\in W}w.
\]
The \emph{spherical subalgebra} is \(\be\rmH_{\hbar,c}^\pol(\fkh,W)\be\).

{Since quantization in this work is stated in terms of \(\hbar\)-adic completions, we also consider the \(\hbar\)-adic completion of the rational Cherednik algebra.
Define the {\it formal rational Cherednik algebra} to be its
\(\hbar\)-adic completion,
\[
\rmH_{\hbar,c}(\fkh,W)
:=
\varprojlim_m
\rmH_{\hbar,c}^{\mathrm{pol}}(\fkh,W)/
\hbar^m\rmH_{\hbar,c}^{\mathrm{pol}}(\fkh,W).
\]
Thus \(\rmH_{\hbar,c}(\fkh,W)\) is topologically generated by
\(\fkh\oplus\fkh^*\) and \(W\), and its relation ideal is the closure of the above polynomial relation ideal.
The {\it formal spherical subalgebra} is defined analogously and
denoted by \(\be \rmH_{\hbar,c}(\fkh, W) \be\).}

\subsection{Dunkl embedding}
For \(y\in\fkh\), define the Dunkl operator by
\begin{equation}\label{eq: dunkl}
	D_y=
	\partial_y+\frac12\sum_{\alpha\in R}c_\alpha
	\langle y,\alpha\rangle\,\alpha^{-1}(s_\alpha-1)
	\in\mD_\hbar(\fkh_{\reg})\rtimes W.
\end{equation}

\begin{proposition}\cite[Proposition~4.5]{EG2002}
Let \(\fkg\) be a finite-dimensional complex reductive Lie algebra with Cartan subalgebra \(\fkh\).
Fix the root data \((R\subset\fkh^*, R^+ \subset R, W)\).
The assignments
	\[
	w\tos w,
	\qquad
	x\tos x,
	\qquad
	y\tos D_y
	\]
	{extend to an injective \(\bbC[[\hbar]]\)-algebra homomorphism}
	\[
	\widetilde\Theta_{\hbar,c}:
	\rmH_{\hbar,c}(\fkh,W)
	\into
	\mD_\hbar(\fkh_{\reg})\rtimes W.
	\]
\end{proposition}

Let \(\rmH_{\hbar,c}^{\mathrm{pol}}(\fkh,W)\) be the polynomial
\(\bbC[\hbar]\)-algebra from Definition~\ref{def: rca}.  First form its
 Ore localization at the powers of \(\delta\), and then define
\[
\rmH_{\hbar,c}(\fkh,W)^\loc
:=
\varprojlim_m
\frac{
\rmH_{\hbar,c}^{\mathrm{pol}}(\fkh,W)[\delta^{-1}]
}{
\hbar^m\rmH_{\hbar,c}^{\mathrm{pol}}(\fkh,W)[\delta^{-1}]
}.
\]

At the uncompleted \(\bbC[\hbar]\)-level, the localized Dunkl embedding is
already an isomorphism.
Its inverse is determined on generators by
\[
w\longmapsto w,
\qquad
x\longmapsto x,
\qquad
\partial_y\longmapsto
y-\frac12\sum_{\alpha\in R}c_\alpha
\langle y,\alpha\rangle\alpha^{-1}(s_\alpha-1).
\]
Indeed, these formulas are mutually inverse on the localized generators.
Thus the localized map and its spherical restriction are isomorphisms
over \(\bbC[\hbar]\).  Taking \(\hbar\)-adic completions gives the
following proposition.

\begin{proposition}[Localized Dunkl isomorphism]
\label{prop: dunkl embed loc neq 0}
Let \(\fkg\) be a finite-dimensional complex reductive Lie algebra with Cartan subalgebra \(\fkh\).
	Fix the root data \((R\subset\fkh^*, R^+ \subset R, W)\).
	The Dunkl embedding and its spherical restriction extend to
	\(\bbC[[\hbar]]\)-algebra isomorphisms
	\[
	\widetilde\Theta_{\hbar,c}^{\loc}:
	\rmH_{\hbar,c}(\fkh,W)^\loc
	\xrightarrow{\sim}
	\mD_\hbar(\fkh_{\reg})\rtimes W
	\]
	and
	\[
	\widetilde\Theta_{\hbar,c}^{\sph}:
	\be\rmH_{\hbar,c}(\fkh,W)^\loc\be
	\xrightarrow{\sim}
	\mD_\hbar(\fkh_{\reg})^W.
	\]
\end{proposition}

The construction below uses \((\widetilde\Theta_{\hbar,c}^{\sph})^{-1}\) to identify invariant differential operators on \(\fkh_{\reg}\) with elements of the localized spherical Cherednik algebra.

At \(\hbar=0\), the Dunkl embedding specializes to
\[
\widetilde\Theta_{0,c}:
\rmH_{0,c}(\fkh,W)^\loc
\longrightarrow
\bbC[\fkh_{\reg}\times\fkh^*]\rtimes W,
\]
with
\[
x\tos x,
\qquad
y\tos D_y^0
=y+\frac12\sum_{\alpha\in R}c_\alpha
\langle y,\alpha\rangle\,\alpha^{-1}(s_\alpha-1),
\qquad
w\tos w.
\]
The following classical form is the specialization of the Dunkl
isomorphism.
\begin{proposition}[{\cite[Proposition~7.9(ii)]{Eti2009Lec}}]
\label{prop: dunkl embed loc 0}
	For a finite-dimensional complex reductive Lie algebra \(\fkg\) with Cartan subalgebra \(\fkh\) and Weyl group \(W\),
	\[
	\rmH^{\loc}_{0,c}(\fkh,W)
	\cong
	\bbC[\fkh_{\reg}\times\fkh^*]\rtimes W,
	\qquad
	\be\rmH^{\loc}_{0,c}(\fkh,W)\be
	\cong
	\bbC[\fkh_{\reg}\times\fkh^*]^W.
	\]
\end{proposition}

\subsection{From quiver to rational Cherednik algebras}

By Proposition \ref{prop: dunkl embed loc neq 0}, the Dunkl isomorphism identifies the localized rational Cherednik algebra with the crossed product of differential operators on \(\fkh_{\reg}\) by \(W\). Its spherical restriction identifies the localized spherical subalgebra with \(\mD_\hbar(\fkh_{\reg})^W\).

From this point onward, we specialize to
	\[
	\fkg=\gl_n(\bbC),
	\qquad
	\fkh=\bbC^n,
	\qquad
	W=S_n.
	\]
The inverse of the  Dunkl isomorphism
\[
(\widetilde{\Theta}^{\loc}_{\hbar,c})^{-1}:
\mD_\hbar(\fkh_{\reg})\rtimes W
\xrightarrow{\sim}
\rmH_{\hbar,c}(n)^\loc
\]
is given on generators by
\begin{equation}\label{for: inverse dunkl in x}
	w\in W\tos w,
	\qquad
	x\in\fkh^*\tos x,
	\qquad
	\partial_y\tos
	y-\frac12\sum_{\alpha\in R}c_\alpha
	\langle y,\alpha\rangle\,\alpha^{-1}(s_\alpha-1).
\end{equation}

We now construct the desired morphism from the quantum preprojective algebra to the localized spherical rational Cherednik algebra. Let
\[
\iota:
\mD_\hbar(\fkh)^{S_n}
\hookrightarrow
\mD_\hbar(\fkh_{\reg})^{S_n}
\]
be the localization map induced by the open embedding
\(\fkh_{\reg}\hookrightarrow\fkh\).

\begin{proposition}
	\label{prop: quant preproj to RCA}
	Let \(Q\) be the Jordan quiver and let \(n\) be a positive integer. Then, for every \(c\in\bbC\), there is a \(\bbC[[\hbar]]\)-algebra morphism
	\[
	\Upsilon_{\hbar,c}:
	\Pi Q_{\hbar}
	\longrightarrow
	\be\rmH_{\hbar,c}(n)^\loc\be,
	\]
	which factors through the radial quantum trace map \(\widetilde{\Tr^q}\).
	In particular, when \(c=0\), the image of \(\Upsilon_{\hbar,0}\) lies in \(\be\rmH_{\hbar,0}(n)\be\).
\end{proposition}

\begin{proof}
	Define \(\Upsilon_{\hbar,c}\) to be the composite
	\[
	\Pi Q_{\hbar}
	\xrightarrow{\widetilde{\Tr^q}}
	\mD_\hbar(\fkh)^{S_n}
	\xrightarrow{\iota}
	\mD_\hbar(\fkh_{\reg})^{S_n}
	\xrightarrow{(\widetilde{\Theta}_{\hbar,c}^{\sph})^{-1}}
	\be\rmH_{\hbar,c}(n)^\loc\be.
	\]
	The radial quantum trace, localization, and inverse  spherical Dunkl isomorphism are \(\bbC[[\hbar]]\)-algebra morphisms. Their composite is therefore a \(\bbC[[\hbar]]\)-algebra morphism, and the asserted factorization is built into its definition.
	
	When \(c=0\), the formula \eqref{for: inverse dunkl in x} reduces to
	\[
	w\tos w,
	\qquad
	x\tos x,
	\qquad
	\partial_y\tos y.
	\]
	Thus the inverse spherical Dunkl isomorphism introduces no root denominators when \(c=0\). Since \(\widetilde{\Tr^q}\) lands in the nonlocalized algebra \(\mD_\hbar(\fkh)^{S_n}\), the image of \(\Upsilon_{\hbar,0}\) lies in \(\be\rmH_{\hbar,0}(n)\be\).
\end{proof}

We next pass to the semiclassical specialization.

\begin{corollary}\label{cor: preproj vs rca}
	Let \(Q\) be the Jordan quiver and let \(n\) be a positive integer. Then, for every \(c\in\bbC\), there is a Poisson algebra morphism
	\[
	\Upsilon_{0,c}:
	\Sym((\Pi Q)_\cyc)
	\longrightarrow
	\be\rmH_{0,c}(n)^\loc\be,
	\]
	which factors through the classical trace map on the preprojective algebra.
	In particular, when \(c=0\), the image of \(\Upsilon_{0,0}\) lies in {\(\be\rmH_{0,0}(n)\be\)}.
\end{corollary}
\begin{proof}
Define the map by the composite
\[
\Sym((\Pi Q)_\cyc)
\xrightarrow{\widetilde{\Tr}}
\bbC[T^*\fkh]^{S_n}
\xrightarrow{\iota_0}
\bbC[\fkh_{\reg}\times\fkh^*]^{S_n}
\xrightarrow{(\widetilde\Theta_{0,c}^{\sph})^{-1}}
\be\rmH_{0,c}(n)^\loc\be,
\]
where \(\iota_0\) is localization. The classical trace and localization
are Poisson. The last map is Poisson because it is the specialization
of the inverse spherical Dunkl isomorphism between flat formal
quantizations, with the same convention
\(\{-,-\}=-\hbar^{-1}[-,-]\bmod\hbar\). For \(c=0\), the inverse
Dunkl formula introduces no root denominators, so the image is
nonlocalized.

Each factor defining
\(\Upsilon_{\hbar,c}
=
(\widetilde\Theta_{\hbar,c}^{\mathrm{sph}})^{-1}
\circ\iota
\circ\widetilde{\operatorname{Tr}^{q}}\)
then specializes to the
corresponding factor of the displayed classical composite,
which proves the claimed specialization compatibility.
\end{proof}

\subsection{Computation of
\texorpdfstring{\(\Upsilon_{\hbar,c}\)}{the Cherednik morphism}}
\label{subsec:inverse-spherical-dunkl-power-sum}

As shown in Proposition \ref{prop: quant preproj to RCA}, the morphism \(\Upsilon_{\hbar,c}\) is obtained from the radial quantum trace by applying the inverse spherical Dunkl isomorphism.
The preceding section expresses \(\widetilde{\Tr^q}\) in power-sum coordinates. We now compute the inverse spherical Dunkl isomorphism in the same coordinate system.

Let
\[
G(p)=\bigl(p_{a+b-2}\bigr)_{1\leq a,b\leq n},
\qquad
V=\bigl(x_i^{a-1}\bigr)_{1\leq a,i\leq n}.
\]
Then
\[
G(p)=VV^{\mathsf T},
\qquad
\det G(p)=(\det V)^2=\delta^2.
\]
Consequently, \(G(p)\) is invertible on \(\fkh_{\reg}\). 
Moreover,
\[
	\bbC[\fkh_{\reg}]^{S_n}
	\cong
	\bbC[p_1,\ldots,p_n,\delta^{-2}],
\]
and each entry of \(G(p)^{-1}\) belongs to this ring.

For \(1\leq a\leq n\), set
\begin{equation}\label{eq:power-sum-Cherednik-derivative}
	\Xi_{p_a}
	:=
	\frac1a\sum_{b=1}^n
	\bigl(G(p)^{-1}\bigr)_{a, b}\,
	\be\left(\sum_{i=1}^n x_i^{b-1}y_i\right)\be
	\ \in\
	\be\rmH_{\hbar,c}(n)^\loc\be.
\end{equation}

\begin{proposition}\label{prop: inverse dunkl power sum}
	Let \(\fkh \subset \gl_n (\bbC)\) be the Cartan subalgebra of diagonal matrices.
	Fix the root data \((R, R^+, W)\).
	The assignment
	\begin{equation}\label{for: inverse dunkl power sum}
		p_a\mapsto p_a\be,
		\qquad
		\delta^{-2}\mapsto\delta^{-2}\be,
		\qquad
		\partial_{p_a}\mapsto\Xi_{p_a},
		\qquad 1\leq a\leq n,
	\end{equation}
	extends uniquely to a \(\bbC[[\hbar]]\)-algebra morphism
	\[
	\Phi_{\hbar,c}:
	\mD_\hbar(\fkh_{\reg})^{S_n}
	\longrightarrow
	\be\rmH_{\hbar,c}(n)^\loc\be.
	\]
Furthermore, it is the inverse of the localized spherical Dunkl isomorphism in Proposition \ref{prop: dunkl embed loc neq 0}.
\end{proposition}

\begin{proof}
	The action of \(S_n\) on \(\fkh_{\reg}\) is free, and the quotient map
\[
\fkh_{\reg}\longrightarrow\fkh_{\reg}/S_n
\]
is finite \'etale. Hence \'etale descent for
differential operators gives
\[
\mD_\hbar(\fkh_{\reg})^{S_n}
\cong
\mD_\hbar(\fkh_{\reg}/S_n).
\]

Since
\[
\bbC[\fkh_{\reg}]^{S_n}
=
\bbC[p_1,\ldots,p_n,\delta^{-2}],
\]
	the uncompleted Weyl algebra of differential operators on
	\(\fkh_{\reg}/S_n\) is generated by
\[
p_1,\ldots,p_n,
\qquad
\delta^{-2},
\qquad
\partial_{p_1},\ldots,\partial_{p_n}.
\]
Its \(\hbar\)-adic completion
\(\mD_\hbar(\fkh_{\reg})^{S_n}\) is topologically generated by the same elements.
It therefore suffices to verify the defining relations for
this uncompleted Weyl algebra and then extend uniquely by
\(\hbar\)-adic continuity.

The generators of this uncompleted Weyl algebra
\[
p_1,\ldots,p_n,
\qquad
\delta^{-2},
\qquad
\partial_{p_1},\ldots,\partial_{p_n},
\]
 satisfy the following relations
	\begin{equation}\label{eq:power-sum-differential-relations}
	\begin{aligned}&
	\relax [p_a,p_b]=[p_a, \delta^{-2}]=0,\quad
	[\partial_{p_a},\partial_{p_b}]=0,
	\quad \delta^2 \delta^{-2} = \delta^{-2} \delta^{2} = 1,
	\\
	& [\partial_{p_a},F]
	  =\,\partial_{p_a}(F),
	  \quad
	  F\in\bbC[\fkh_{\reg}]^W.
	\end{aligned}
	\end{equation}
	To prove the multiplicative property, 
	we only need to check: for every
	\(F\in\bbC[\fkh_{\reg}]^W\),
	\begin{equation}\label{eq:Cherednik-power-sum-Weyl-relation}
	  [ \Xi_{p_a} ,F \be]
	  =
	  \partial_{p_a}(F) \be,
	  \qquad
	  [ \Xi_{p_a} ,\Xi_{p_b}]
	  =
	  0.
	\end{equation}
	
	We prove the first equality.
	Recall that  Dunkl operators are
	\[
	  D_i
	  =
	  \partial_{x_i}
	  +
	  c\sum_{j\neq i}(x_i-x_j)^{-1}(s_{ij}-1).
	\]
	Since \(F\) is \(S_n\)-invariant, it commutes with every \(s_{ij}\).
	Therefore
	\[
	\begin{aligned}
	  {}[D_i,F]
	  &=
	  [\partial_{x_i},F]
	  +
	  c\sum_{j\neq i}
	  \left[
	    (x_i-x_j)^{-1}(s_{ij}-1),
	    F
	  \right]                                                     \\
	  &=
	  \partial_{x_i}(F).
	\end{aligned}
	\]
	Using \eqref{eq:power-sum-Cherednik-derivative}, the fact that
	\(F\) commutes with \(\be\), and the preceding identity, we obtain
	\[
	\begin{aligned}
	  {}[ \Xi_{p_a} ,F\be]
	  &=
	  \frac{1}{a}\sum_{b=1}^n
	  \bigl(G(p)^{-1}\bigr)_{a, b}
	  \left[
	    \be\left(\sum_{i=1}^n x_i^{b-1}y_i\right)\be,
	    F\be
	  \right]                                                     \\
	  &=
	  \frac{1}{a}\sum_{b=1}^n
	  \bigl(G(p)^{-1}\bigr)_{a, b}
	  \be\left(
	    \sum_{i=1}^n x_i^{b-1}[y_i,F]
	  \right)\be                                                  \\
	  &=
	  \frac{1}{a}\sum_{b=1}^n
	  \bigl(G(p)^{-1}\bigr)_{a, b}
	  \be\left(
	    \sum_{i=1}^n x_i^{b-1}\partial_{x_i}(F)
	  \right)\be.
	\end{aligned}
	\]
	The chain rule gives
	\[
	\begin{aligned}
	  \sum_{i=1}^n x_i^{b-1}\partial_{x_i} (F)
	  &=
	  \sum_{i=1}^n x_i^{b-1}
	  \sum_{k=1}^n
	  kx_i^{k-1}\partial_{p_k}(F )                              \\
	  &=
	  \sum_{k=1}^n
	  k\left(\sum_{i=1}^n x_i^{b+k-2}\right)
	  \partial_{p_k}(F)                                          \\
	  &=
	  \sum_{k=1}^n
	  k p_{b+k-2}\partial_{p_k}(F)                               \\
	  &=
	  \sum_{k=1}^n
	  kG(p)_{b, k}\partial_{p_k}(F).
	\end{aligned}
	\]
	It follows that
	\[
	\begin{aligned}
	  {}[ \Xi_{p_a} ,F\be]
	  &=
	  \frac{1}{a}
	  \sum_{b,k=1}^n
	  \bigl(G(p)^{-1}\bigr)_{a, b}
	  kG(p)_{b, k}\partial_{p_k}(F)\be                             \\
	  &=
	  \frac{1}{a}
	  \sum_{k=1}^n
	  k\delta_{a, k}\partial_{p_k}(F)\be                           \\
	  &=
	  \,\partial_{p_a}(F)\be.
	\end{aligned}
	\]
	This proves the first identity in
	\eqref{eq:Cherednik-power-sum-Weyl-relation}. 
	
	For the second one,
	we reduce the problem to checking that \(\widetilde{\Theta}_{\hbar,c}^{\sph}
	\bigl( \Xi_{p_a} \bigr)
	 = \partial_{p_a}\),
	 since \(\widetilde{\Theta}_{\hbar,c}^{\sph}\) is an isomorphism, under which one derives \([\Xi_{p_a}, \Xi_{p_b}] = 0\) from \([\partial_{p_a}, \partial_{p_b}] = 0\). 
	In fact, 
	for \(0\leq r\leq n-1\), the chain rule yields
	\[
	  \sum_{i=1}^n x_i^r\partial_{x_i}
	  =
	  \sum_{a=1}^n
	  a p_{a+r-1}\partial_{p_a}.
	\]
	It gives
	\begin{equation}\label{eq:inverse-Hankel-vector-fields}
	  \partial_{p_a}
	  =
	  \frac{1}{a}\sum_{b=1}^n
	  \bigl(G(p)^{-1}\bigr)_{a, b}
	  \sum_{i=1}^n x_i^{b-1}\partial_{x_i}.
	\end{equation}
	Since \((s_{ij}-1)\be=0\), for \(0\leq r\leq n-1\) one has
	\begin{equation}\label{eq:spherical-image-power-sum-vector-field}
	  \widetilde\Theta_{\hbar,c}^{\sph}
	  \left(
	    \be\left(\sum_{i=1}^n x_i^r y_i\right)\be
	  \right)
	  =
	  \sum_{i=1}^n x_i^r\partial_{x_i}.
	\end{equation}
	Combining \eqref{eq:power-sum-Cherednik-derivative},
	\eqref{eq:inverse-Hankel-vector-fields}, and
	\eqref{eq:spherical-image-power-sum-vector-field}, we compute
	\[
	\begin{aligned}
	\widetilde{\Theta}_{\hbar,c}^{\sph}
	\bigl( \Xi_{p_a} \bigr)
	&=
	\widetilde{\Theta}_{\hbar,c}^{\sph}
	\left(
	\frac{1}{a}
	\sum_{b=1}^n
	\bigl(G(p)^{-1}\bigr)_{a, b}
	\be
	\left(
	\sum_{i=1}^n x_i^{b-1}y_i
	\right)
	\be
	\right)\\
	&=
	\frac{1}{a}
	\sum_{b=1}^n
	\bigl(G(p)^{-1}\bigr)_{a, b}
	\widetilde{\Theta}_{\hbar,c}^{\sph}
	\left(
	\be
	\left(
	\sum_{i=1}^n x_i^{b-1}y_i
	\right)
	\be
	\right)\\
	&=
	\frac{1}{a}
	\sum_{b=1}^n
	\bigl(G(p)^{-1}\bigr)_{a, b}
	\sum_{i=1}^n x_i^{b-1}\partial_{x_i}\\
	&=
	\partial_{p_a}.
	\end{aligned}
	\]
	Thus
	\eqref{for: inverse dunkl power sum} extends to an algebra morphism.

	Finally, under the standard identification
	\[
	\be\bigl(\mD_\hbar(\fkh_{\reg})\rtimes S_n\bigr)\be
	\cong
	\mD_\hbar(\fkh_{\reg})^{S_n},
	\]
	we have
	\begin{equation}\label{eq:spherical-dunkl-on-functions}
	\widetilde{\Theta}_{\hbar,c}^{\sph}(p_a\be)=p_a,
	\qquad
	\widetilde{\Theta}_{\hbar,c}^{\sph}(\delta^{-2}\be)=\delta^{-2}.
	\end{equation}
	The above calculations show
	that the composite of \(\Phi_{\hbar, c}\) defined by
	\eqref{for: inverse dunkl power sum} with
	\(\widetilde\Theta_{\hbar,c}^{\sph}\) is the identity on \(\mD_\hbar(\fkh_{\reg})^{S_n}\). Since
	\(\widetilde\Theta_{\hbar,c}^{\sph}\) is an isomorphism by
	Proposition~\ref{prop: dunkl embed loc neq 0}, the constructed morphism \(\Phi_{\hbar,c}\)
	is precisely
	\(\bigl(\widetilde\Theta_{\hbar,c}^{\sph}\bigr)^{-1}\).
\end{proof}

\begin{remark}[Dependence on parameter \(c\)]
	\label{rem:inverse-spherical-dunkl-c-dependence}
	The assignment \eqref{for: inverse dunkl power sum}
	contains no explicit occurrence of \(c\). This does not mean that
	\(\Phi_{\hbar,c}\) is independent of \(c\): for each value of \(c\), its target is the algebra
	\[
	\be\rmH_{\hbar,c}(n)^\loc\be,
	\]
	whose multiplication depends on \(c\) through the defining relations.
\end{remark}

The following theorem computes the images of the quantum trace operators.

\begin{theorem}[Explicit source elements under the inverse spherical Dunkl map]
\label{thm: Upsilon in power-sum}
Let \(Q\) be the Jordan quiver and let \(n\) be a positive integer.
Fix \(c\in\bbC\) and the root data \((R,R^+,W)\). Then, for
\(s\in\bbZ_{\geq0}\) and \(r\in\bbZ_{\geq1}\),
\begin{equation}\label{eq:image-quantum-trace-inverse-dunkl-power-sum}
	\Upsilon_{\hbar,c}(\overline w_{s,r})
	=
				\sum_{(\pi,\alpha)}\hbar^{r - |\pi|}
				C_{\pi,\alpha}^{s,r}(p)\be
				\prod_{B\in\pi}
				\left(
				\Xi_{p_{\alpha(B)}}
				-
				\frac{1}{2}
				\partial_{p_{\alpha(B)}}(\ln (\delta^2))\be
				\right).
		\end{equation}
For \(r=0\), one has
\[
	\Upsilon_{\hbar,c}(\overline w_{s,0})
	=\Phi_{\hbar,c}\bigl(\hc(\Tr^q(X^s))\bigr)=p_s\be.
\]
\end{theorem}

\begin{proof}
By the definition of \(\Upsilon_{\hbar,c}\) and the notation for the
quotient class,
\[
\widetilde{\Tr^q}(\overline w_{s,r})
=
\hc\!\left(\Tr^q(w_{s,r})\right)
=
\hc\!\left(\Tr^q(X^sD^r)\right).
\]
It is therefore enough to apply
Proposition~\ref{prop: inverse dunkl power sum} to
Proposition~\ref{prop: hc formula power sums}.  For \(r=0\),
\[
\widetilde{\Tr^q}(\overline w_{s,0})
=
\hc\!\left(\Tr^q(X^s)\right)
=p_s.
\]
\end{proof}


\appendix

\section{Multivariate Fa\`a di Bruno formula}\label{app: fdb}

We recall the multivariate Fa\`a di Bruno formula of Constantine--Savits \cite{ConSav1996Mul}.
Throughout this section,
\(D_{x_i}\), \(D_{y_a}\), \(D_{u_j}\), and \(D_{v_a}\)
denote the ordinary partial derivatives
\(\frac{\partial}{\partial x_i}\),
\(\frac{\partial}{\partial y_a}\),
\(\frac{\partial}{\partial u_j}\), and
\(\frac{\partial}{\partial v_a}\), respectively.

Let
\[
f=f(y_1,\ldots,y_m)
\]
be a polynomial function on \(\bbC^m\). Let
\[
g=(g^{(1)},\ldots,g^{(m)}):\bbC^d\longrightarrow\bbC^m
\]
be a polynomial map, where
\[
g^{(a)}=g^{(a)}(x_1,\ldots,x_d),
\qquad a=1,\ldots,m.
\]
We denote the composition by
\[
h=f\cc g,
\qquad
h(x_1,\ldots,x_d)
=
f\bigl(g(x_1,\ldots,x_d)\bigr).
\]

For
\(\nu=(\nu_1,\ldots,\nu_d)\in\bbZ_{\geq 0}^d,\)
{we use the  following multi-index conventions}
\[
	D_x^\nu
	=
	D_{x_1}^{\nu_1}\cdots D_{x_d}^{\nu_d},
	\qquad
	\nu!=\nu_1!\cdots \nu_d!.
\]
For a multi-index \(\ell\in\bbZ_{\geq 0}^d\), write
\[
	g_\ell^{(a)}=(D_x^\ell g^{(a)})(t_0),
	\qquad
	g_\ell=(g_\ell^{(1)},\ldots,g_\ell^{(m)}).
\]
For \(k=(k_1,\ldots,k_m)\in\bbZ_{\geq 0}^m\), put
\[
	[g_\ell]^k=\prod_{a=1}^m (g_\ell^{(a)})^{k_a}.
\]
Fix a total order \(\prec\) on \(\bbZ_{\geq0}^d\) for which \(0\) is the
least element.

For \(1\leq s\leq |\nu|\), let \(P_s(\nu,\beta)\) be the set of all
\[
	(k_1,\ldots,k_s;\ell_1,\ldots,\ell_s)
\]
such that
\begin{align*}
&k_j\in\bbZ_{\geq 0}^m,
\qquad |k_j|>0,
\qquad j=1,\ldots,s,\\
&0\prec \ell_1\prec\cdots\prec\ell_s,
\qquad \ell_j\in\bbZ_{\geq 0}^d,\\
&\sum_{j=1}^s k_j=\beta,\qquad
\sum_{j=1}^s |k_j|\ell_j=\nu.
\end{align*}
Here \(|k_j|=k_{j,1}+\cdots+k_{j,m}\).

With the above conventions,
the multivariate Fa\`a di Bruno formula,
given in
Constantine--Savits \cite{ConSav1996Mul},
is as follows.

\begin{proposition}[{\cite[Theorem 2.1]{ConSav1996Mul}}]
For \( |\nu|=0\), one has \(D_x^0h(t_0)=f(g(t_0))\).
For \( |\nu|\geq1\), one has
	\begin{equation}\label{for: CS}
	D_x^\nu h(t_0)
	=
	\sum_{1\leq |\beta|\leq |\nu|}
	(D_y^\beta f)(g(t_0))
	\sum_{s=1}^{|\nu|}
	\sum_{P_s(\nu,\beta)}
	\nu!
	\prod_{j=1}^s
	\frac{[g_{\ell_j}]^{k_j}}{k_j!\,[\ell_j!]^{|k_j|}}.
	\end{equation}
\end{proposition}

We now specialize \eqref{for: CS} to the form used in this paper.  Take
\[
	d=r,
	\qquad
	m=n,
	\qquad
	x=u,
	\qquad
	t_0=0,
\]
and let
\[
	f=F(v_1,v_2,\ldots,v_n),
	\qquad
	g^{(a)}=\lambda_a(u_1,u_2,\ldots,u_r),
	\qquad a=1,\ldots,n.
\]
Thus
\[
	h(u_1,\ldots,u_r)=F(\lambda_1(u),\ldots,\lambda_n(u)).
\]
Finally, take the multi-index
\begin{equation}\label{eq:squarefree-nu}
	\nu=(1,1,\ldots,1)\in\bbZ_{\geq 0}^r.
\end{equation}
Then
\[
	D_u^\nu=D_{u_1}\cdots D_{u_r},
	\qquad
	|\nu|=r,
	\qquad
	\nu!=1.
\]
 {Substituting these choices into \eqref{for: CS} gives}
\begin{equation}\label{eq:CS-specialized}
D_{u_1}\cdots D_{u_r}F(\lambda(u))\big|_{u=0}
=
\sum_{1\leq |\beta|\leq r}
(D_v^\beta F)(\lambda(0))
\sum_{s=1}^{r}
\sum_{P_s(\nu,\beta)}
\prod_{j=1}^s
\frac{[g_{\ell_j}]^{k_j}}{k_j!\,[\ell_j!]^{|k_j|}},
\end{equation}
where now
\[
	[g_{\ell_j}]^{k_j}
	=
	\prod_{a=1}^n
	\left(D_u^{\ell_j}\lambda_a(0)\right)^{k_{j,a}}.
\]

 { {We now rewrite \eqref{eq:CS-specialized} in a more compact form.}}  Let
\[
	[r]=\{1,\dots,r\},
	\qquad
	[n]=\{1,\dots,n\}.
\]
For a subset \(B\subseteq [r]\), let \(I_B\in\bbZ_{\geq 0}^r\) be its
indicator multi-index:
\[
	(I_B)_i=
	\begin{cases}
	1, & i\in B,\\
	0, & i\notin B.
	\end{cases}
\]
Define
\(D_u^B
=
\prod_{i\in B}D_{u_i}\).
If \(B=\varnothing\), we set
\(D_u^\varnothing = \id\).
For example, if \(B=\{2,5,7\}\), then
\[
	D_u^B
	=
	D_{u_2}D_{u_5}D_{u_7}.
\]

A \emph{set partition} of \([r]\) is a collection
\[
	\pi=\{B_1,\dots,B_\ell\}
\]
of nonempty, pairwise disjoint subsets \(B_j\subseteq [r]\) such that
\[
	B_1\sqcup\cdots\sqcup B_\ell=[r].
\]
The subsets \(B_j\) are the \emph{blocks} of \(\pi\).  
Let
$\Pi([r])$
denote the set of all set partitions of \([r]\).
For a partition \(\pi\in\Pi([r])\), a map
\[
	\alpha:\pi\to[n]
\]
assigns to each block \(B\in\pi\) an index
\(\alpha(B)\in\{1,\dots,n\}\).
For such a pair
\((\pi,\alpha)\), define the \(v\)-derivative operator
\[
	D_v^\alpha
	=
	\prod_{B\in\pi}D_{v_{\alpha(B)}}.
\]
For example, if
\[
	\pi=\{B_1,B_2,B_3\},
	\qquad
	\alpha(B_1)=2,
	\quad
	\alpha(B_2)=2,
	\quad
	\alpha(B_3)=5,
\]
then
\[
	D_v^\alpha F
	=
	D_{v_2}^2D_{v_5}F.
\]

\begin{theorem}\label{thm: FDB}
With the above notation, \eqref{eq:CS-specialized} is equivalent to the
following formula:
\begin{equation}\label{Brunosformula}
\begin{aligned}
&\left.
D_{u_1}D_{u_2}\cdots D_{u_r}
F\bigl(\lambda_1(u),\dots,\lambda_n(u)\bigr)
\right|_{u=0}\\
&=
{
\sum_{\pi\in\Pi([r])}
\sum_{\alpha:\pi\to[n]}
\left(
	\prod_{B\in\pi}
	\left(D_u^B\lambda_{\alpha(B)}\right)(0)
\right)
(D_v^\alpha F)(\lambda(0)).
}
\end{aligned}
\end{equation}
\end{theorem}

\begin{proof}
Let \(\varepsilon_a\in\bbZ_{\geq0}^n\) denote the \(a\)-th standard basis vector.  We compare the summation data in the specialized Constantine--Savits formula \eqref{eq:CS-specialized} with pairs \((\pi,\alpha)\), where \(\pi\in\Pi([r])\) and \(\alpha:\pi\to[n]\).

Take an element
\[
	(k_1,\ldots,k_s;\ell_1,\ldots,\ell_s)\in P_s(\nu,\beta)
\]
for the square-free multi-index \(\nu=(1,\ldots,1)\).  For each \(t\in[r]\), the defining condition
\(\sum_{j=1}^s |k_j|\ell_j=\nu\) gives
\[
	1=\nu_t=\sum_{j=1}^s |k_j|(\ell_j)_t.
\]
Since each \(|k_j|\) is a positive integer and each \((\ell_j)_t\) is a nonnegative integer, this identity implies that every nonzero entry of every \(\ell_j\) is equal to \(1\), and that no two distinct \(\ell_j\)'s have a nonzero entry in the same coordinate.  Moreover, \(0\prec\ell_j\) implies that each \(\ell_j\) is nonzero.  Hence, for every \(j\), the coefficient \(|k_j|\) must be \(1\).  Therefore \(k_j=\varepsilon_{a_j}\) for a unique \(a_j\in[n]\).

Let
\[
	B_j=\operatorname{supp}(\ell_j)=\{t\in[r]:(\ell_j)_t=1\}.
\]
The preceding paragraph shows that the subsets \(B_1,\ldots,B_s\) are nonempty, pairwise disjoint, and have union \([r]\).  Thus they form a set partition \(\pi\) of \([r]\).  The order
\(0\prec\ell_1\prec\cdots\prec\ell_s\) is not additional data in the set-partition formula; it is only the unique ordering of the blocks obtained by listing their indicator multi-indices according to \(\prec\).  Define \(\alpha:\pi\to[n]\) by
\[
	\alpha(B_j)=a_j,
	\qquad
	k_j=\varepsilon_{a_j}.
\]
Then \(\beta=\sum_{j=1}^s \varepsilon_{a_j}\).  Equivalently, \(\beta_a\) is the number of blocks \(B\in\pi\) for which \(\alpha(B)=a\).

Conversely, start with a pair \((\pi,\alpha)\).  List the blocks of \(\pi\) as \(B_1,\ldots,B_s\) in the unique order satisfying
\[
	0\prec I_{B_1}\prec\cdots\prec I_{B_s}.
\]
Set
\[
	\ell_j=I_{B_j},
	\qquad
	k_j=\varepsilon_{\alpha(B_j)},
	\qquad
	\beta=\sum_{j=1}^s \varepsilon_{\alpha(B_j)}.
\]
Then \((k_1,\ldots,k_s;\ell_1,\ldots,\ell_s)\in P_s(\nu,\beta)\).  These two constructions are inverse to each other.  Hence the Constantine--Savits summation over \(s\), \(\beta\), and \(P_s(\nu,\beta)\) is the same as the set-partition summation over \(\pi\) and \(\alpha\).

It remains to compare the summands.  Under the above bijection, \(|k_j|=1\), \(k_j!=1\), and \(\ell_j!=1\), because \(\ell_j=I_{B_j}\) is square-free.  Moreover,
\[
	[g_{\ell_j}]^{k_j}
	=
	g_{I_{B_j}}^{(\alpha(B_j))}
	=
	\left(D_u^{B_j}\lambda_{\alpha(B_j)}\right)(0).
\]
The outer derivative also agrees: since \(\beta_a\) counts the blocks mapped to \(a\),
\[
	(D_v^\beta F)(\lambda(0))
	=
	\left(
	\prod_{B\in\pi}D_{v_{\alpha(B)}}F
	\right)(\lambda(0))
	=
	(D_v^\alpha F)(\lambda(0)).
\]
Substituting these identities into \eqref{eq:CS-specialized} gives exactly the displayed set-partition formula.
\end{proof}

\section{Examples of Computing Coefficients}\label{sect:examples}

\begin{example}[The complete coefficient for \texorpdfstring{\(r=1\)}{r=1}]
	Let
\[
   r=1,\qquad
   \pi=\{B\},\qquad
   B=\{1\},\qquad
   \alpha(B)=a,\qquad
   s\geq0.
\]
Since \(\alpha(B)\in\{1,\ldots,n\}\), one has \(a\geq1\).
The set \(\Ord(\pi)\) contains only the ordering
\(\bfi^B=(1)\).  The admissible exponent data are the pairs
\[
   (\ell_0^B,\ell_1^B)\in\bbZ_{\geq0}^2,
   \qquad
   \ell_0^B+\ell_1^B=a-1.
\]
Equivalently, they are
\[
   (\ell_0^B,\ell_1^B)=(u,a-1-u),
   \qquad 0\leq u\leq a-1,
\]
so there are exactly \(a\) admissible pairs.

The graph has vertex set \(\{0,1\}\).  Its distinguished edge is
\(\{0,1\}\), while the block \(B\) contributes the same underlying edge
\(\{1,0\}\), because the cyclic successor of \(1\) is \(0\) when \(r=1\).
Thus the graph has the single connected component
\[
   C=\{0,1\}.
\]
For every admissible pair, the vertex weights are
\[
   \lambda_0=s+\ell_0^B,
   \qquad
   \lambda_1=\ell_1^B.
\]
Consequently,
\[
   L_C
   =\lambda_0+\lambda_1
   =s+\ell_0^B+\ell_1^B
   =s+a-1,
\]
and every admissible pair contributes \(p_{s+a-1}\).  Summing the \(a\)
graph contributions gives
\[
   C_{\pi,\alpha}^{s,1}(p)
   =
   \sum_{\substack{\ell_0^B,\ell_1^B\geq0\\
                   \ell_0^B+\ell_1^B=a-1}}
   p_{s+a-1}
   =
   a\,p_{s+a-1}.
\]
This graph computation recovers the coefficient used in the first-order
proposition in Section~\ref{subsec: hc power sum formula}, without
evaluating a matrix-index sum.
\end{example}

\begin{example}[The graph for the \(r=4\) data]
\label{Appexample:r=4}
Take
\[
   r=4,\qquad s=2,
\]
and
\[
   \pi=\{B_1,B_2,B_3\},
   \qquad
   B_1=\{1\},\quad
   B_2=\{2,4\},\quad
   B_3=\{3\}.
\]
Assume \(n\geq3\), and let
\[
   \alpha(B_1)=1,\qquad
   \alpha(B_2)=3,\qquad
   \alpha(B_3)=1.
\]
The singleton blocks have the fixed orderings
\[
   \bfi^{B_1}=(1),
   \qquad
   \bfi^{B_3}=(3),
\]
whereas the two orderings of \(B_2\) are
\[
   \bfi^{B_2}=(2,4),
   \qquad
   \bfi^{B_2}=(4,2).
\]

The distinguished edge is \(\{0,1\}\), and the singleton blocks contribute
\[
   B_1:\ \{1,2\},
   \qquad
   B_3:\ \{3,4\}.
\]
For the ordering \((2,4)\), the block \(B_2\) contributes
\[
   \{2,0\},
   \qquad
   \{4,3\},
\]
where the first edge uses the fact that the cyclic successor of \(4\) is
\(0\).  For the ordering \((4,2)\), the same block contributes
\[
   \{4,3\},
   \qquad
   \{2,0\};
\]
here the second edge uses the cyclic successor of \(4\).
Thus the two orderings produce the same underlying undirected edge set.
The edge \(\{3,4\}\) is contributed by both \(B_2\) and \(B_3\), but its
repetition does not change the underlying graph.  For either ordering the
connected components are
\[
   C_1=\{0,1,2\},
   \qquad
   C_2=\{3,4\}.
\]
\end{example}

\begin{example}[The complete \(r=4\) coefficient]
\label{Appexample:r=4contd}
We continue with the data and the two orderings in
Example~\ref{Appexample:r=4}.  For each singleton block, admissibility
forces all exponent data to be zero.  For \(B_2\), it gives
\[
   \ell_0^{B_2}+\ell_1^{B_2}+\ell_2^{B_2}=1,
\]
so the three admissible choices are
\[
   (1,0,0),\qquad
   (0,1,0),\qquad
   (0,0,1).
\]
Together with the two orderings of \(B_2\), these give six graph
contributions.
\begin{center}
\small
\renewcommand{\arraystretch}{1.2}
\begin{tabular}{c|c|c|c}
\(\bfi^{B_2}\) &
\((\ell_0^{B_2},\ell_1^{B_2},\ell_2^{B_2})\) &
\((L_{C_1},L_{C_2})\) &
power-sums \\
\hline
\((2,4)\) & \((1,0,0)\)  & \((2,1)\) & \(p_1p_2\) \\
\((2,4)\) & \((0,1,0)\)  & \((3,0)\) & \(p_0p_3\) \\
\((2,4)\) & \((0,0,1)\)  & \((2,1)\) & \(p_1p_2\) \\
\((4,2)\) & \((1,0,0)\)  & \((3,0)\) & \(p_0p_3\) \\
\((4,2)\) & \((0,1,0)\)  & \((2,1)\) & \(p_1p_2\) \\
\((4,2)\) & \((0,0,1)\)  & \((3,0)\) & \(p_0p_3\)
\end{tabular}
\end{center}
The three choices for the ordering \((2,4)\) therefore contribute
\[
   2p_1p_2+p_0p_3,
\]
and the three choices for the ordering \((4,2)\) contribute
\[
   p_1p_2+2p_0p_3.
\]
Summing all six graph contributions yields the complete coefficient
\[
   C_{\pi,\alpha}^{2,4}(p)
   =
   3p_1p_2+3p_0p_3
   =
   3p_1p_2+3np_3.
\]
\end{example}

\bibliographystyle{plain}
\bibliography{bib_ncquantandRCA}

\end{document}